\documentclass[sn-basic,Numbered]{sn-jnl}% Basic Springer Nature Reference Style/Chemistry Reference Style
\usepackage{rounakstyle}
\numberwithin{equation}{section}
\makeatletter
\def\moverlay{\mathpalette\mov@rlay}
\def\mov@rlay#1#2{\leavevmode\vtop{%
		\baselineskip\z@skip \lineskiplimit-\maxdimen
		\ialign{\hfil$\m@th#1##$\hfil\cr#2\crcr}}}
\newcommand{\charfusion}[3][\mathord]{
	#1{\ifx#1\mathop\vphantom{#2}\fi
		\mathpalette\mov@rlay{#2\cr#3}
	}
	\ifx#1\mathop\expandafter\displaylimits\fi}
\makeatother

\newcommand{\cupdot}{\charfusion[\mathbin]{\cup}{\cdot}}

\begin{document}

\title[Ising model on preferential attachment models]{Ising model on preferential attachment models}


%\author[1]{\fnm{Rajat} \sur{Hazra} \orcidlink{0000-0002-9254-4763}}\email{r.s.hazra@math.leidenuniv.nl}

\author[1]{\fnm{Remco} \sur{van der Hofstad} \orcidlink{0000-0003-1331-9697}}\email{r.w.v.d.hofstad@TUE.nl}

\author[2]{\fnm{Rounak} \sur{Ray} \orcidlink{0000-0001-8219-9778}}\email{ray.rounak@outlook.com}

%\affil[1]{\orgdiv{Department of Mathematics}, \orgname{University of Leiden}, \orgaddress{\country{The Netherlands}}}

\affil[1]{\orgdiv{Department of Mathematics and Computer Science}, \orgname{Eindhoven University of Technology}, \orgaddress{\country{The Netherlands}}}

\affil[2]{\orgdiv{Probability and Statistics of Discrete Structures}, \orgname{Simons Laufer Mathematical Sciences Institute}, \orgaddress{\country{The United States of America}}}

%\affil[3]{\orgdiv{Department}, \orgname{Organization}, \orgaddress{\street{Street}, \city{City}, \postcode{610101}, \state{State}, \country{Country}}}


\abstract{We study the Ising model on affine preferential attachment models with general parameters. We identify the thermodynamic limit of several quantities, arising in the large graph limit, such as pressure per particle, magnetisation, and internal energy for these models. Furthermore, for $m\geq 2$, we determine the inverse critical temperature for preferential attachment models as $\beta_c(m,\delta)=0$ when $\delta\in(-m,0]$, while, for $\delta>0$,
\[
	\beta_c(m,\delta)= \arctanh\left\{ \frac{\delta}{2\big( m(m+\delta)+\sqrt{m(m-1)(m+\delta)(m+\delta+1)} \big)} \right\}~.
\]
Our proof for the thermodynamic limit of pressure per particle critically relies on the belief propagation theory for factor models on locally tree-like graphs, as developed by Dembo, Montanari, and Sun. It has been proved that preferential attachment models admit the P\'{o}lya point tree as their local limit under general conditions. We use the explicit characterisation of the P\'{o}lya point tree and belief propagation for factor models to obtain the explicit expression for the thermodynamic limit of the pressure per particle. Next, we use the convexity properties of the internal energy and magnetisation to determine their thermodynamic limits. To study the phase transition, we prove that the inverse critical temperature for a sequence of graphs and its local limit are equal. Finally, we show that $\beta_c(m,\delta)$ is the inverse critical temperature for the P\'{o}lya point tree with parameters $m$ and $\delta$, using results from Lyons who shows that the critical inverse temperature is closely related to the percolation critical threshold. This part of the proof heavily relies on the critical percolation threshold for P\'{o}lya point trees established earlier with Hazra.
}

%%================================%%
%% Sample for structured abstract %%
%%================================%%




\maketitle
%%%%%%%%%%%%%%%%%%%%%%%%%%%%%%%%%%%%%%%%%%%%%%%%%%%
% !TEX root = ising.tex
\section{Introduction}\label{sec:introduction}
The Ising model, originally introduced to describe ferromagnetism in statistical mechanics, represents a system of spins on a graph, where each spin can take one of two states. The interactions between neighbouring spins determine the system's overall magnetic properties. This model has since found applications beyond physics, including in sociology, biology, neuroscience, and finance, where it helps to explain the collective behaviour of interconnected entities. The Ising model has been extensively studied in the literature; see, for example, \cite{G06}. Recently, the focus in this field has extended to the study of the Ising model on random graphs. Dembo and Montanari studied the Ising model on graphs that are locally unimodular branching processes in \cite{DM10}, restricting their analysis to cases where the average degree has finite variance. The Ising model on Erd\H{o}s-R\'{e}nyi graphs in both the zero and high-temperature regimes was studied in \cite{GdS08}. The Ising model on configuration models has been analysed in \cite{DM10,DGvdH10}. Dembo, Montanari and Sun also studied factor models on locally tree-like graphs, shedding light on the Ising model in these contexts in \cite{DMS13}.

Real-life random networks often possess power-law degree distributions. Barabási and Albert introduced a dynamic random graph model exhibiting a power-law degree distribution in \cite{Barabasi'99}. A generalisation of this model leads to a vast class of preferential attachment models, see e.g. \cite{vdH2,BergerBorgs,DEH09,DM13}. In \cite{GHHR22}, together with Hazra and extending \cite{BergerBorgs}, we proved that a large class of affine preferential attachment models with a linear attachment function converges locally to the P\'{o}lya point tree, introduced in \cite{BergerBorgs}. In this article, we study the quenched Ising model on preferential attachment models.

\subsection{Ising model on finite graphs}\label{subsec:intro:def-ising}
In this section, we define the Ising model on a sequence of (possibly random) finite graphs. Let $(G_n)_{n\geq 1}$ be a sequence of such graphs, where $G_n=(V_n,E_n)$ has vertex set $V_n=\{1,2,\ldots,n\}\equiv[n]$, and some (possibly random) set of edges $E_n$. Each vertex in $G_n$ is assigned an Ising spin value in $\{-1,+1\}$. A spin configuration of $G_n$ is denoted by ${\boldsymbol{\sigma}}=(\sigma_1,\ldots,\sigma_n)\in\{-1,+1\}^n$. For any $\beta\geq 0$ and $B_v\in\R$ for all $v\in V_n$, the Ising model on $G_n$ is given by the Boltzmann distribution
\begin{equation}\label{eq:def:boltzmann:seq}
	\mu_n(\boldsymbol{\sigma})=\frac{1}{Z_n(\beta,\underline{B})}\exp\left\{ \beta\sum\limits_{\{u,v\}\in E_n}\sigma_u\sigma_v +\sum\limits_{v\in V_n}B_v\sigma_v \right\},
\end{equation}
where $Z_n({\beta,\underline{B}})$ is the normalising constant, also known as the \emph{partition function}, given by
\begin{equation}\label{eq:def:partition-function}
	Z_n({\beta,\underline{B}})=\sum\limits_{\boldsymbol{\sigma}\in\{-1,+1\}^n}\exp\left\{ \beta\sum\limits_{\{u,v\}\in E_n}\sigma_u\sigma_v +\sum\limits_{v\in V_n}B_v\sigma_v \right\}.
\end{equation}
For any function $f:\{-1,+1\}^n\to \R,$ define $\langle f(\boldsymbol{\sigma}) \rangle_{\mu_n}$ as
\begin{equation}\label{eq:def:quenched-expectation}
	\langle f(\boldsymbol{\sigma}) \rangle_{\mu_n} =\sum\limits_{\boldsymbol{\sigma}\in\{-1,+1\}^n} f(\boldsymbol{\sigma})\mu_n(\boldsymbol{\sigma}).
\end{equation}
Lastly, we define the thermodynamic quantity of central interest, the so-called \emph{pressure per particle}, as
\begin{equation}\label{eq:def:pressure-per-particle:initial}
	\psi_n(\beta,\underline{B})=\frac{1}{n}\log Z_n(\beta,\underline{B}).
\end{equation}
In this paper, we shall mainly consider a fixed external magnetic field acting on all the vertices of the graph, resulting in $B_v=B$ for all $v\in V_n$. Note that, under the assumption of a fixed magnetic field, the Boltzmann distribution function is symmetric about the sign of $B$. Thus, without loss of generality, we assume $B>0$, and denote the partition function and pressure per particle as $Z_n(\beta,B),$ and
\begin{equation}\label{eq:def:pressure-per-particle}
	\psi_n(\beta,{B})=\frac{1}{n}\log Z_n(\beta,{B}),
\end{equation}
respectively. We shall study the thermodynamic limit of this pressure per particle in great detail for preferential attachment models.
%---------------------------------------------
\subsection{The random graph models}\label{subsec:intro:models}
In this paper, we focus on sequential preferential attachment models that have the P\'{o}lya point trees as their local limit, specifically models (a), (b) and (d) as defined in \cite{vdH1}, excluding their tree cases. We fix $m\in\N\setminus{\{1\}}$ and $\delta>-m$. In these models, every new vertex is introduced to the existing graph with $m$ edges incident to it. The models are defined by their edge-connection probabilities.

We start from any finite graph with $2$ vertices, and finitely many connections between them, such that at least one of the initial vertices has degree at most $m$. Let $a_1$ and $a_2$ denote the degrees of the vertices $1$ and $2$, respectively. Without loss of generality, consider $a_2\leq m$. We also allow for self-loops in the initial graph.

\medskip\noindent
\textbf{Model (a).} For each new vertex $v$ joining the graph and $j=[m]$, the attachment probabilities are given by
\eqn{\label{def:model:a}
	\prob\Big( v\overset{j}{\rightsquigarrow} u\mid\PA^{(a)}_{v,j-1}(m,\delta) \Big)= \begin{cases}
		\frac{d_u(v,j-1)+\delta}{c_{v,j}^{(a)}}\hspace{1.15cm}\text{for}~u<v,\\
		\frac{d_v(v,j-1)+{j\delta}/{m}}{c_{v,j}^{(a)}}\hspace{0.65cm}\text{for}~u=v,
	\end{cases}
}
where $v\overset{j}{\rightsquigarrow} u$ denotes that vertex $v$ connects to $u$ with its $j$-th edge, $\PA^{(a)}_{v,j}({m},\delta)$ denotes the graph on $v$ vertices, \FC{with} the $v$-th vertex \FC{having already paired its first} $j$ out-edges, and $d_u(v,j)$ denoting the degree of vertex $u$ in $\PA^{(a)}_{v,j}({m},\delta)$. We identify $\PA^{(a)}_{v+1,0}({m},\delta)$ with $\PA^{(a)}_{v,m}({m},\delta)$. The normalizing constant $c_{v,j}^{(a)}$ in \eqref{def:model:a} equals
\eqn{\label{eq:normali}
	c_{v,j}^{(a)} := a_{[2]}+2\delta+(2m+\delta)(v-3)+2(j-1)+1+\frac{j\delta}{m} \, ,
}
where $a_{[2]}=a_1+a_2$. We denote the above model by $\PA^{(a)}_{v}({m},\delta)$, which is equivalent to $\PA_v^{(m,\delta)}(a)$ as defined in \cite{vdH1}.

\medskip
\begin{remark}[Universality across preferential attachment models]
	\rm Although we state the main theorem for model (a) described above, the result holds true for models (b) and (d), described in \cite{vdH1}, as well. To reduce notational complexity, we focus on model (a) here. \hfill$\blacksquare$
\end{remark}
\medskip
\begin{remark}[Marked local convergence]
	\rm The P\'olya point tree is the marked local limit of a large class of affine preferential attachment models. We refer the reader to \cite[Section~1.3]{GHHR22}, for a detailed definition of marked local convergence. \hfill $\blacksquare$
\end{remark}
		\subsection{The P\'{o}lya point tree}\label{sec:RPPT}
		Berger, Borgs, Chayes and Saberi \cite{BergerBorgs} proved that the P\'olya point tree ($\PPT$) is the local limit of our preferential attachment model. In this section, {we define the $\PPT$ that will act as the vertex-marked local limit of our preferential attachment graphs. In \cite{GHHR22}, together with Hazra, we have proved that the random P\'olya point tree ($\RPPT$) is the vertex-marked local limit of a wide class of preferential attachment models with i.i.d.\ out-degrees. Further, the P\'olya point tree is a special case of the random P\'olya point tree. We adapt the definition of the $\PPT$ from \cite{BergerBorgs} and \cite{GHHR22}. We start by defining the vertex set of this PPT:}

\medskip
		\begin{Definition}[Ulam-Harris set and its ordering]\label{def:ulam}
			{\rm Let $\N_0=\N\cup\{0\}$. {The {\em Ulam-Harris set}} is
				\begin{equation*}
					\mathcal{N}=\bigcup\limits_{n\in\N_0}\N^n.
				\end{equation*}
				For $x = x_1\cdots x_n\in\N^n$ and $k\in\N$, we let \(n\) be its length, denote {the element $x_1\cdots x_nk$ by $xk\in\N^{n+1},$ and call it the \(k\)-th child of \(x\).} The~{\em root} of the Ulam-Harris set is denoted by $\emp\in\N^0$.
				%For any $x\in \mathcal{U}$, we say that $x$ has length $n$ if $x\in\N^n$. 
				{The lexicographic \textit{ordering}} between the elements of the Ulam-Harris set {is as follows:}
				\begin{itemize}
					\item[(a)] for any two elements $x,y\in\mathcal{U}$, {$x>y$ when the length of $x$ is more than that of $y$;}
					\item[(b)] if $x,y\in \N^n$ for some $n$, then $x>y$ if there exists $i\leq n,$ such that $x_j=~y_j~\forall j<i$ and $x_i>~y_i$.\hfill$\blacksquare$
			\end{itemize}}
		\end{Definition}

		We use the elements of the Ulam-Harris set to identify nodes in a rooted tree, since the notation in Definition~\ref{def:ulam} allows us to denote the relationships between children and parents.%, where for $x\in\mathcal{U}$, we denote the $k$-th child of $x$ by the element $xk$. 
		%Note that $\PPT$ is a specific case of $\RPPT$ with deterministic out-degree distribution. So for defining $\PPT$, we adapt the definition of $\RPPT$ from \cite{GHHR22}.
		
		\paragraph{\bf P\'olya point tree ($\PPT$).}
		The  P\'olya point tree, $\PPT(m,\delta)$ is an {\em infinite multi-type rooted random tree}, where $m$ and $\delta>-m$ are the parameters of our preferential attachment model. It is a multi-type branching process, with a mixed continuous and discrete type space. We now describe its properties one by one.
		
		\paragraph{\bf Descriptions of the distributions and parameters used.}
		\begin{enumerate}
			\item[{\btr}] {Define} $\chi = \frac{m+\delta}{2m+\delta}$.
			\smallskip
			\item[{\btr}] {Let} $\Gamma_{\sf{in}}(m)$ denote a Gamma distribution with parameters $m+\delta$ and $1$.
			%\smallskip
			%\item[{\btr}] {Let} $\Gamma_{\sf{in}}^\prime(m)$ denote the size-biased distribution of $\Gamma_{\sf{in}}(m)$, which is also a Gamma distribution with parameters $m+\delta+1$ and $1$.
			%\smallskip
			%\item[{\btr}] {For $\delta>-\infsupp(M),$ let}  $M^{(\delta)}$ be an $\N$-valued random variable such that $$\prob\left(M^{(\delta)}=m\right) = \frac{m+\delta}{\E[M]+\delta}\prob(M=m),$$ i.e., $M^{(\delta)}+\delta$ is a size-biased version of $M+\delta$.
			%\smallskip
			%\item[{\btr}] {In particular,} $M^{(0)}$ is the size-biased distribution of $M$.
		\end{enumerate}
		\smallskip
		
		\paragraph{\bf Feature of the vertices of $\PPT$.}
		Below, to avoid confusion, we use `node' for a vertex in the PPT and `vertex' for a vertex in the PAM. We now discuss the properties of the nodes in $\PPT(m,\delta)$. Every node \({\omega}\), except the root in the $\PPT$, has
		\begin{enumerate}
			%\item[{\btr}] a {\em label} $\omega$ in the Ulam-Harris set $\mathcal{N}$ (recall Definition~\ref{def:ulam});
			\smallskip
			\item[{\btr}] an {\em age} $A_{\omega}\in[0,1]$;
			\smallskip
			\item[{\btr}] a positive number $\Gamma_{\omega}$ called its {\em strength};
			\smallskip
			\item[{\btr}] a {label in $\{{\Old},{\Young}\}$} depending on the age of the {node} and its parent, {with $\Young$ denoting that the node is younger than its parent and $\Old$ denoting that the node is older than its parent.}
		\end{enumerate}

		Based on its label being $\Old$ or $\Young$, every {node} $\omega$ has a number $m_{-}(\omega)$ associated to it. If $\omega$ has label $\Old$, then $m_{-}(\omega)=m$, and $\Gamma_{\omega}$ is distributed as $\Gamma_{\sf{in}}( m+1 )$,
		while if $\omega$ has label $\Young$, then $m_{-}(\omega)=m-1$, and given $m_{-}(\omega), \Gamma_{\omega}$ is distributed as $\Gamma_{\sf{in}}( m )$. The type-space of the P\'{o}lya point tree is given by
		\eqn{\label{eq:state-space:PPT}
		\Scal=[0,1]\times\{\Old,\Young\}\cup[0,1]~.}
		
		\smallskip
		%\begingroup
		\allowdisplaybreaks
		\paragraph{\bf Construction of the PPT.} {We next use the above definitions to construct the PPT using an {\em exploration process}.}
		The root is special in the tree. It has label $\emp$ and its age $A_\emp$ is an independent uniform random variable in $[0,1]$. The root $\emp$ has no label in $\{\Old,\Young\}$, and we let $m_{-}(\emp)=m$.
		Then the {children of the root in the} P\'{o}lya point tree {are} constructed as follows:
		%\endgroup	
		\begin{enumerate}
			\item Sample $U_1,\ldots,U_{m_{-}(\emp)}$ uniform random variables on $[0,1]$, independent of the rest;  
			\item To nodes $\emp1,\ldots, \emp m_{-}(\emp),$ assign the ages $U_1^{1/\chi} A_\emp,\ldots,U_{m_{-}(\emp)}^{1/\chi} A_\emp$ and label $\Old$;
			\item Assign ages $A_{\emp(m_{-}(\emp)+1)},\ldots, A_{\emp(m_{-}(\emp)+\din_\emp)}$ to nodes $\emp(m_{-}(\emp)+1),\ldots, \emp(m_{-}(\emp)+\din_\emp)$. These ages are the \FC{occurrence times} given by a conditionally independent Poisson point process on $[A_\emp,1]$ defined by the random intensity
			\eqn{\label{def:rho_emp}
				\rho_{\emp}(x) = {(1-\chi)}{\Gamma_\emp}\frac{x^{-\chi}}{A_\emp^{1-\chi}},}
			and $\din_\emp$ being the total number of points of this process. Assign label $\Young$ to them;
			\item Draw an edge between $\emp$ and each of $\emp 1,\ldots, \emp(m_{-}(\emp)+\din_\emp)$; 
			\item Label $\emp$ as explored and nodes $\emp1,\ldots, \emp(m_{-}(\emp)+\din_\emp)$ as unexplored.
		\end{enumerate}	
		\smallskip
		Then, recursively over the elements in the set of unexplored nodes, we perform the following breadth-first exploration:
		\smallskip
		\begin{enumerate}
			\item Let $\omega$ denote the smallest currently unexplored node in the Ulam-Harris ordering;
			\item Sample $m_{-}(\omega)$ i.i.d.\ random variables $U_{\omega1},\ldots,U_{\omega m_{-}(\omega)}$ independently from all the previous steps and from each other, uniformly on $[0,1]$. To nodes $\omega 1,\ldots,\omega m_{-}(\omega)$, assign the ages $U_{\omega1}^{1/\chi}A_\omega,\ldots,U_{\omega m_{-}(\omega)}^{\FC{1/\chi}}A_\omega$ and label $\Old$, and set them unexplored;
			\item Let $A_{\omega(m_{-}(\omega)+1)},\ldots,A_{\omega(m_{-}(\omega)+\din_\omega)}$ be the random $\din_\omega$ points given by a conditionally independent Poisson process on $[A_{\omega},1]$ with random intensity
			\eqn{
				\label{for:pointgraph:poisson}
				\rho_{\omega}(x) = {(1-\chi)}{\Gamma_\omega}\frac{x^{-\chi}}{A_\omega^{1-\chi}}.
			}
			Assign these ages to $\omega(m_{-}(\omega)+1),\ldots,\omega(m_{-}(\omega)+\din_\omega)$. {Assign them label $\Young$,} and set them unexplored;
			\item Draw an edge between $\omega$ and each one of the nodes $\omega 1,\ldots,\omega(m_{-}(\omega)+\din_\omega)$;
			\item Set $\omega$ as explored.
		\end{enumerate}
		\smallskip
		We call the resulting tree the {\em P\'olya point tree with parameters $m$ and $\delta$,} and denote it by $\PPT(m,\delta)$. Occasionally we drop $m$ and $\delta$ when referring to $\PPT(m,\delta)$. The vertex-marks of $\PPT$ here are different from the ones in \cite{BergerBorgs}, but we can always retrieve these vertex-marks in the definition in \cite{BergerBorgs} from the vertex-marks used here and the other way round also.

For any node \(\omega\) in the P\'{o}lya point tree, \(\bff(\omega)\) denote its vertex-mark in \(\Scal\). For example, \(\bff(\emp)=U\), whereas \(\bff(\emp 1)=(U_1^{1/\chi}U,\Old)\), where \(U,U_1\) are i.i.d.\, \(\Unif[0,1]\) random variables. 
%the \(\Old\)-labelled child of \(\emp,~\emp1\) has  .
%---------------------------------------------
\subsection{Main result}\label{sebsec:intro:main-results}
In this section, we describe the main results of this paper. Our first theorem identifies the explicit thermodynamic limit of the pressure per particle for preferential attachment models:

\medskip

\begin{Theorem}[Thermodynamic limit of the pressure]\label{thm:thermodynamic-limit:pressure}
	Fix $m\geq 2$ and $\delta>0$, and consider $\PA_n(m,\delta)$ (any of the three models (a), (b), and (d) described in Section~\ref{subsec:intro:models}). Then for all $0\leq \beta<\infty$ and $B\in\R$, the thermodynamic limit of the pressure exists and is deterministic:
	\eqn{\label{eq:thm:thermodynamic-limit:pressure}
		\psi_n(\beta,B)\overset{\prob}{\to} \varphi(\beta,B)~,}
	where $\varphi(\beta,B)$ is a constant.
	The thermodynamic limit of the pressure satisfies $\varphi(\beta,B)=\varphi(\beta,-B)$ for $B>0$, and $\varphi(\beta,0)=\lim\limits_{B\searrow 0}\varphi(\beta,B)$. For $B>0$, \(\varphi(\beta,B)\) is given by
	\eqan{\label{eq:thm:thermodynamic-limit:pressure:explicit}
		& \frac{\E[D(\emp)]}{2}\log \cosh(\beta)-\frac{{\E[D(\emp)]}}{2}\E\Big[\log\big\{1+\tanh(\beta) \tanh(\h(\bff(\hat{\emp} 1)))\tanh(\h(\bff(\hat{\emp}))) \big\}\Big]\\
		&\hspace{0.03cm}+\E\Big[ \log\Big( \e^B\prod\limits_{i=1}^{D(\emp)}\big\{1+\tanh(\beta)\tanh(\h(\bff(\emp i)))\big\}+\e^{-B}\prod\limits_{i=1}^{D(\emp)}\big\{1-\tanh(\beta)\tanh(\h(\bff(\emp i)))\big\} \Big) \Big],\nn	
	}
	where $\hat{\emp}=(U,\Young)$ and $U\sim\Unif[0,1]$, and $\{\h(\omega):\omega\in \Scal\}$ are independent copies of the fixed point functional $\h^\star(\omega)=\h^\star(\omega,\beta,B)$ of the following distributional fixed point equation
	\eqn{\label{eq:distributional-recursion:1}
		\h(\omega)\overset{d}{=}B+\sum\limits_{i=1}^{D(v)}\arctanh \Big(\tanh(\beta)\tanh\big(\h(\bff(v i))\big)\Big)~,}
	where $D(v)$ is the degree of the vertex \(v\) such that \(\bff(v)=\omega\) in the P\'olya point tree and \(vi\) are the children of the node \(v\) in the P\'olya point tree and \(\{\h(\bff(v i)):i\in[D(v)]\}\) are independent of \(D(v)\).
\end{Theorem}

Next, we provide the thermodynamic limits of thermodynamic quantities, such as internal energy and magnetisation for the Ising model on preferential attachment models:

\medskip

\begin{Theorem}[Thermodynamic quantities]\label{thm:limit:thermodynamic-quantities}
	Let $(G_n)_{n\geq 1}$ be a sequence of preferential attachment graphs with parameters $m$ and $\delta$. Fix $\beta\geq 0$ and $B\in\R$, 
	\begin{enumerate}
		\item[(a)]\label{thm:magnetization} {\bf Magnetisation.}
		Let $M_n(\beta,B)=\frac{1}{n}\langle\sum_{i\in[n]}\sigma_i\rangle_{\mu_n}$ be the magnetisation per vertex. Then, its thermodynamic limit exists and is given by
		\eqn{\label{eq:def:magnetization}
			\frac{\partial}{\partial B}\psi_n(\beta,B)=M_n(\beta,B)\overset{\prob}{\to}M(\beta,B)=\frac{\partial}{\partial B}\varphi(\beta,B)=\E[\tanh(\h(\bff(\emp)))]~.
		}
		\item[(b)]\label{thm:internal:energy} {\bf Internal energy.}
		Let $U_n(\beta,B)=-\frac{1}{n} \langle\sum_{(u,v)\in E_n} \sigma_u\sigma_v\rangle_{\mu_n}$ denote the internal energy per vertex. Then, its thermodynamic limit exists and is given by
		\eqan{\label{eq:def:internal:energy}
			U_n(\beta,B)\overset{\prob}{\to}U(\beta,B)=&-\frac{\partial}{\partial\beta}\varphi(\beta,B)\nn\\
			=&-m\E\left[ \frac{\tanh(\beta)+\tanh(\h(\bff(\hat{\emp})))\tanh(\h(\bff(\hat{\emp} 1)))}{1+\tanh(\beta)\tanh(\h(\bff(\hat{\emp})))\tanh(\h(\bff(\hat{\emp} 1)))} \right]~.
		}
	\end{enumerate}
\end{Theorem}

Next, we define $\beta_c(m,\delta)$ as the Ising inverse critical temperature for a sequence of preferential attachment models with parameters $m$ and $\delta$, defined as
\eqn{\label{eq:def:inverse-critical-temperature:PA}
	\beta_c=\inf\Big\{\beta:\lim\limits_{B\searrow 0}M(\beta,B)>0\Big\}~.}
The final theorem of this paper addresses the inverse critical temperature for preferential attachment models. Here, we explicitly identify the almost sure inverse critical temperature for the preferential attachment model:

\medskip 

\begin{Theorem}[Ising inverse critical temperature]\label{thm:inv:critical-temp:PA}
	Fix $m\geq 2$ and $\delta>0$, and let $\beta_c(m,\delta)$ be the Ising inverse critical temperature for the P\'{o}lya point tree with parameters $m$ and $\delta$. Then, for $\delta>0$,
	\eqn{\label{eq:thm:inv:critical-temp:PA}
		\beta_c(m,\delta)= \arctanh\left\{ \frac{\delta}{2\big( m(m+\delta)+\sqrt{m(m-1)(m+\delta)(m+\delta+1)} \big)} \right\},}
	whereas for $\delta\in(-m,0],~\beta_c(m,\delta)=0$ almost surely.
\end{Theorem}

\medskip

\paragraph{\bf Novelty.} We identify the explicit expression for the pressure per particle for the preferential attachment model in the Ising model settings in \eqref{eq:thm:thermodynamic-limit:pressure:explicit}, and show that it equals the expression in \cite[Theorem~1.9]{DMS13}. Alongside, we identify the thermodynamic limits of related Ising quantities, and compute the explicit inverse critical temperature for Ising models on preferential attachment models.

\medskip

\paragraph{Organisation of the article.}
The remainder of the article is organised as follows: In Section~\ref{sec:preliminary:results}, we state some preliminary lemmas, and we refer to external sources for their proofs. In Section~\ref{sec:thermodynamic-limit:subsec:belief-propagation}, we explain the concept of belief propagation, that lays the foundation of the proof of Theorem~\ref{thm:thermodynamic-limit:pressure}. In Section~\ref{sec:thermodynamic-limit:subsec:convergence:pressure}, we prove Theorem~\ref{thm:thermodynamic-limit:pressure}. Next, in Section~\ref{sec:thermodynamic-limit:subsec:convergence:thermodynamic-quantity}, we prove Theorems~\ref{thm:thermodynamic-limit:pressure} and \ref{thm:limit:thermodynamic-quantities} as a consequence of local convergence and Theorem~\ref{thm:thermodynamic-limit:pressure}. Finally, in Section~\ref{sec:inv:critical-temperature}, we prove Theorem~\ref{thm:inv:critical-temp:PA}.

%As a direct consequence of this theorem, we can identify 
%--------------------------------------------------------

%%%%%%%%%%%%%%%%%%%%%%%%%%%%%%%%%%%%%%%%%%%%%%
\section{Thermodynamic limits on the preferential attachment models}\label{sec:thermodynamic:limit}

In this section, we prove the existence  of the thermodynamic limits of Ising quantities, such as pressure per particle, magnetisation, and internal energy. We obtain the thermodynamic limit of the pressure using the \emph{belief propagation} method from \cite{DMS13}. By applying \cite[Theorem~1.9]{DMS13}, we establish Theorem~\ref{thm:thermodynamic-limit:pressure}, which provides the explicit form of the thermodynamic limit. Next, in Section~\ref{sec:thermodynamic-limit:subsec:convergence:thermodynamic-quantity}, we prove Theorem~\ref{thm:limit:thermodynamic-quantities} by utilising the convexity properties of these quantities.

%%%%%%%%%%%%%%%%%%%%%%%%%%%%%%%%%%%%%%%%%%%%%%%%%%
\subsection{Preliminary results}\label{sec:preliminary:results}
We now state several preliminary results that we frequently use in this section. These lemmas hold true under very mild conditions, and we verify that these conditions are met in our case. We do not provide proofs for these lemmas but instead cite the articles where they are proven in detail:

\medskip
\begin{lemma}[GKS inequality]\label{lem:GKS}
	Consider two Ising measures $\mu$ and $\mu^\prime$ on graphs $G = (V, E)$ and $G^\prime= (V, E^\prime)$, with inverse temperatures $\beta$ and $\beta^\prime$, and external fields $\underline{B}$ and $\underline{B}^\prime$, respectively. If $E\subseteq E^\prime$, $\beta \leq \beta^\prime$, and $0 \leq B_i \leq B^\prime_i$ for all $i \in V$, then, for any $U \subset V$,
	\begin{equation}\label{eq:l1em:GKS}
		0\leq \left\langle{\prod\limits_{i\in U}\sigma_i}\right\rangle_{\mu}\leq \left\langle{\prod\limits_{i\in U}\sigma_i}\right\rangle_{\mu^\prime}.
	\end{equation}
\end{lemma}
%\ch{\cite{DGvdH10} provides the relevant citations.}
In \cite{G67}, this result is proved in a restricted settings and \cite{KS68} generalised this result. 
The following lemma simplifies the computation of the Ising measure on a tree by reducing it to the computation of Ising measures on subtrees:

%The following lemma simplifies the computation of the Ising measure on a tree by reducing it to the computation of Ising measures on subtrees:
\medskip
\begin{lemma}[Tree pruning]\label{lem:tree-pruning}
	For $U$ a subtree of a finite tree $T$, let $\partial U$ be the subset of vertices in $U$ that connect to a vertex in $W \equiv T \setminus U$. Denote by $\langle \sigma_u \rangle_{\mu_{W,u}}$ the magnetisation of vertex $u \in \partial U$ in the Ising model on $W \cup \{u\}$. Then, the marginal Ising measure on $U$, $\mu_U^{T}$, is equivalent to the Ising measure on $U$ with the magnetic fields
	\begin{equation}\label{eq:lem:tree-pruning}
		B_u^\prime =\begin{cases}
			\arctanh\big( \langle \sigma_u \rangle_{\mu_{W,u}} \big), & u \in \partial U,\\
			B_u, & u \in U \setminus \partial U.
		\end{cases}
	\end{equation}
\end{lemma}
%Proof to Lemma~\ref{lem:tree-pruning} can be obtained in 
\noindent For more details, see \cite{DM10,DGvdH10,RvdHSF}.
%%%%%%%%%%%%%%%%%%%%%%%%%%%%%%%%%%%%%%%%%%%%%%

%-----------------------------------------------------------
\subsection{Belief propagation on trees}\label{sec:thermodynamic-limit:subsec:belief-propagation}
%-----------------------------------------------------------
For a rooted tree $\Tree$ and $t\in\mathbb{N}$, let $\Tree_t$ denote the tree $\Tree$ pruned at height $t$, and let $\nu_\Tree^{t,+}$ denote the root marginal for the Ising model on $\Tree_t$ with $+$ boundary condition, i.e.\, all the leaves in \(\Tree_t\) have spin \(+1\). Similarly, we define $\nu_\Tree^{t,f}$ as the root marginal for the Ising model on $\Tree_t$ with free boundary condition, i.e.\, leaf-spins are i.i.d.\, \(\Unif\{-1,+1\}\). From \cite[Lemma~4.1]{DMS13}, we know that the limits $\nu_\Tree^{t,+}$ and $\nu_\Tree^{t,f}$ exist as $t\to\infty$, and we denote these limiting distributions as $\nu_\Tree^{+}$ and $\nu_\Tree^{f}$, respectively.

Let $\Tree_{x\to y}$ denote the subtree of $\Tree$ obtained by deleting the edge $\{x , y\}$ and rooting it at $x$. For a P\'{o}lya point tree $\Tree$ and $j\in[D(\emp)]$, $\Tree_{\emp\to\emp j}$ represents the subtree of the P\'{o}lya point tree rooted at $\emp$ with the component hanging from $\emp j$ deleted (including $\emp j$). On the other hand, $\Tree_{\emp j\to\emp }$ is a P\'{o}lya point tree rooted at $\emp j$. Essentially, for any $j\in[D(\emp)]$,
\eqn{\label{eq:BP:def:1}
	\Tree=\Tree_{\emp \to \emp j}\cupdot \Tree_{\emp j\to\emp}~,
}
where $\cupdot$ represents the disjoint union operator. We now define the following root marginals:
\eqn{\label{eq:root-marginal}
	%\begin{split}
	\nu_{\emp\to\emp j}^{\dagger}\equiv \nu_{\Tree_{\emp\to\emp j}}^{\dagger}\quad\text{and}\quad\nu_{\emp j\to\emp}^{\dagger}\equiv \nu_{\Tree_{\emp j\to\emp}}^{\dagger}~,
}
%\end{split}
for $\dagger\in\{+,f\}$.
By \cite[Lemma~5.15]{RvdHSF},
\[
\nu_\Tree^{+}=\nu_\Tree^{f},~\quad\mu~\text{almost surely}.
\]
Let us denote $\magn{\to\emp j}=2\nu_{\emp\to\emp j}^{+}(+1)-1$, and $\magn{\emp j\to}=2\nu_{\emp j\to\emp}^{+}(+1)-1$. Note that $\magn{\to \emp j}$ is the root magnetisation in $\Tree_{\emp\to\emp j}$. Similarly, $\magn{ \emp j\to }$ is the root magnetisation in the P\'{o}lya point tree $\Tree_{\emp j\to\emp}$.

%\RR{Change this definition maybe?Alongside $\{\h(\omega):\omega\in\Scal\}$ defined in Theorem~\ref{thm:thermodynamic-limit:pressure}, we define for any $i\in\N,~\{\h_{\sss -i}(\omega):\omega\in \Scal\}$ to be independent copies of the fixed point functional $\h_{\sss -i}^\star(\omega)=\h_{\sss -i}^\star(\omega,\beta,B)$ of the distributional recursion
%\eqn{\label{eq:distributional-recursion:2}
%	\h_{\sss -i}^{(t+1)}(\omega)\overset{d}{=}B+\sum\limits_{\substack{j=1\\j\neq i}}^{D_t(\omega)}\arctanh \Big(\tanh(\beta)\tanh\big(\h^{(t)}(\omega i)\big)\Big)~,}
%where $\h^{(0)}(\omega)=B$ for all $\omega\in\Scal,$ and $\big(D_t(\omega)\big)_{t\geq 1}$, are i.i.d. random variables with distribution $D(\omega)$, and
%$\h^{(t)}(\omega)$ are independent of $D_t(\omega)$. Note that, all these functional are defined on P\'{o}lya point tree.}
%\begin{lemma}\label{lem:h:magnetization}
%	For any $j\in[D(\emp)]$,
%	\eqan{
%		&\h(\emp j)=\arctanh\big( \magn{\emp j\to} \big)~,\label{eq:prop:root-magnetization:1}\\
%		\mbox{and}\quad &\h_{\sss -j}(\emp)\equiv \arctanh\big( \magn{\to \emp j} \big)\label{eq:prop:root-magnetization:2}}
%	are the unique fixed point solutions to \eqref{eq:distributional-recursion:1} and \eqref{eq:distributional-recursion:2}, respectively.
%\end{lemma}
\medskip
The root magnetisation of any locally-finite tree satisfies the following fixed point equation:

\medskip
\begin{lemma}\label{lem:h:magnetization}
	Let $\Tree$ be a multi-type branching process with a general type-space $\Sbold$, and assume that $\Tree$ is almost surely locally finite. Furthermore, label the nodes of $\Tree$ using the Ulam-Harris notation. Consider a distributional functional $\h$ as a fixed point solution to the following recursion:
	\begin{equation}\label{eq:test:1}
		\h(\bff(\omega)) = B + \sum\limits_{i=1}^{D(\omega)} \arctanh\Big( \tanh(\beta)\tanh\big( \h(\bff(\omega i)) \big) \Big),
	\end{equation}
	where $D(\omega)$ is the degree of the node $\omega$ in $\Tree$, for some $B > 0$ and for all $\omega \in \Sbold$ and $\ell \geq 0$. Then, the root magnetisation  $\magn{\omega\to}$ of $\Tree$, rooted at $\omega$, under the inverse temperature $\beta$ and external magnetic field $B > 0$, for all $\bff(\omega) \in \Sbold$, satisfies
	\begin{equation}\label{eq:prop:root-magnetization:1}
		\h(\bff(\omega)) = \magn{\omega\to},\quad \text{a.s.}
	\end{equation}
\end{lemma}
\medskip
By the definition of the P\'{o}lya point tree, it can be shown to be locally finite. Then, Lemma~\ref{lem:h:magnetization} proves that for the P\'{o}lya point tree,
\begin{equation}\label{eq:test:2}
	\h(\bff(\emp j)) = \magn{\emp j\to},\quad \text{for all}~j\in [D(\emp)]~.
\end{equation}
Further, for any $j\in[D(\emp)]$, let $\h_{\sss -j}(\bff(\emp))$ be the fixed point solution to the recursion relation in \eqref{eq:test:1} for $\Tree_{\emp\to\emp j}$. Since $\Tree$ is locally finite, and $\Tree_{\emp\to\emp j}$ is a subgraph of $\Tree,~\Tree_{\emp\to\emp j}$ as well. Therefore, by Lemma~\ref{lem:h:magnetization},
\begin{equation}\label{eq:test:3}
	\h_{\sss -j}(\bff(\emp)) = \magn{\to\emp j}~.
\end{equation}

The fact that \eqref{eq:prop:root-magnetization:1} is a fixed point solution to \eqref{eq:distributional-recursion:1} follows directly from \cite[Proof of Proposition~5.13]{RvdHSF}. The other cases follow in a similar manner. Since this proof is a minor adaptation of \cite{RvdHSF} to our settings, we defer this proof to {Appendix~\ref{appendix:results}}.
%\RR{Do I need to provide the proof to this lemma? Or stating just this enough?}

%------------------------------------------------------------------
\subsection{Proof of Theorem~\ref{thm:thermodynamic-limit:pressure}}\label{sec:thermodynamic-limit:subsec:convergence:pressure}
%------------------------------------------------------------------
With Lemma~\ref{lem:h:magnetization} established, we now start with the proof of Theorem~\ref{thm:thermodynamic-limit:pressure}. We use the line of proof as used in \cite[Proof of Theorem~1.9]{DMS13}. We first sketch the proof outline below. 
\paragraph{\bf Overview of the proof}
By the fundamental theorem of calculus, for any \(B_1>B_0\),
\eqn{\label{eq:proof:outline:1}
	\lim\limits_{n\to\infty} \Big[ \psi_n(\beta,B_1)-\psi_n(\beta,B_0) \Big] =\lim\limits_{n\to\infty}\int\limits_{B_0}^{B_1}\frac{\partial}{\partial B}\psi_n(\beta,B)\sdv{B}~.
}
Further, differentiating \(\psi_n(\beta,B)\) with respect to \(B\), we obtain
\eqn{\label{eq:proof:outline:2}
	\pdv{}{B}\psi_n(\beta,B)=\frac{1}{n}\sum\limits_{v\in[n]} \anglbrkt{\sigma_u}_{\mu_n}=\E\Big[ \anglbrkt{\sigma_{o_n}}_{\mu_n}\mid G_n \Big]~,
}
where \(o_n\) is a \emph{uniformly chosen} vertex from the vertex set \([n]\), and the expectation is taken with respect to the law of \(o_n\), while \(\mu_n\) denotes the Boltzmann distribution on \(G_n\). By the GKS inequality in Lemma~\ref{lem:GKS}, for every \(\ell>0\),
\eqn{\label{eq:proof:outline:3}
	\anglbrkt{\sigma_{o_n}}_{B_{o_n}(\ell)}^{f}\leq \anglbrkt{\sigma_{o_n}}_{\mu_n} \leq \anglbrkt{\sigma_{o_n}}_{B_{o_n}(\ell)}^{+}~,
}
where \(\anglbrkt{\sigma_{o_n}}_{B_{o_n}(\ell)}^{+/f}\) is the root magnetization in the Ising model on \(B_{o_n}(\ell)\) with \(+/{\rm{free}}\) boundary conditions on \(\nbhd{B_{o_{n}}(\ell)}\). By the local convergence of preferential attachment models to the P\'{o}lya point tree \cite{BergerBorgs,GHHR22}, we obtain that for any \(\ell\in\N\) fixed, the distribution of \(B_{o_n}(\ell)\) converges in probability to that of \(B_\emp(\ell)\), where \(B_\emp(\ell)\) is the P\'{o}lya point tree rooted at \(\emp\), explored up to and including \(\ell\)-th generation. Therefore,
\eqn{\label{eq:proof:outline:4}
	\E\Big[\anglbrkt{\sigma_{o_n}}_{B_{o_n}(\ell)}^{+/f}\mid G_n\Big] \overset{\prob}{\longrightarrow} \E\Big[\anglbrkt{\sigma_{\emp}}_{B_{\emp}(\ell)}^{+/f}\Big]~.
}
Here the expectation in the RHS of \eqref{eq:proof:outline:4} is with respect to the randomness of the P\'{o}lya point tree. The remainder of the proof follows in three steps. First, \cite[Lemma~3.1]{DGvdH10} proves that both \(\anglbrkt{\sigma_{o_n}}_{B_{o_n}(\ell)}^+\) and \(\anglbrkt{\sigma_{o_n}}_{B_{o_n}(\ell)}^f\) converge to the \emph{same} limit as \(\ell\to\infty\).
Next, we identify this limiting quantity as \(\pdv{}{B}\bar\varphi(\beta,B)\), and by \eqref{eq:proof:outline:4}, \(\pdv{}{B}\bar\varphi(\beta,B)\) is bounded, by the following proposition:

\medskip
\begin{Proposition}[Magnetisation limit]\label{prop:mag:lim}
	For all \(\beta,B>0\),
	\eqn{\label{eq:prop:mag:lim}
		\pdv{}{B}\psi_n(\beta,B)\overset{\prob}{\longrightarrow}\E\Big[ \anglbrkt{\sigma_\emp}_\mu \Big] = \pdv{}{B}\bar\varphi(\beta,B)~,
	}
	where \(\bar\varphi(\beta,B)\) is defined as
	\begin{align*}
			\bar\varphi(\beta,B) = &\ \mathbb{E}_{\mu}\left[ \log\left\{ \sum\limits_{\sigma\in\{-1,1\}} \exp\left(B\sigma\right) \prod\limits_{j=1}^{D(\emp)} \left( \sum\limits_{\sigma_j\in\{-1,1\}} \exp\left(\beta\sigma\sigma_j\right) \nu_{\emp j\to\emp}(\sigma_j) \right) \right\}\right.\nonumber \\
		&\hspace{1.75cm}\left.-\frac{1}{2}\sum\limits_{j=1}^{D(\emp)} \log\left\{ \sum\limits_{\sigma,\sigma_j\in\{-1,1\}} \exp\left(\beta\sigma\sigma_j\right) \nu_{\emp j\to\emp}(\sigma_j) \nu_{\emp\to\emp j}(\sigma) \right\} \right]~.
	\end{align*}
\end{Proposition}
By the dominated convergence theorem, we can interchange the limit and integral in \eqref{eq:proof:outline:1}, so that we perform integration to obtain
\eqn{\label{eq:proof:outline:5}
\lim\limits_{n\to\infty}\psi_n(\beta,B_1)-\psi_n(\beta,B_0) = \bar\varphi(\beta,B_1)-\bar\varphi(\beta,B_0)~.
}
Next, we prove the following proposition, identifying the explicit formula for \(\bar\varphi(\beta,B)\) as the RHS of \eqref{eq:thm:thermodynamic-limit:pressure:explicit}.

\medskip
\begin{Proposition}[Identifying explicit \(\bar\varphi(\beta,B)\)]\label{prop:explicit:pressure}
	For any \(\beta,B\geq 0\),
	\[
		\bar\varphi(\beta,B) =\varphi(\beta,B)~,
	\]
	where \(\varphi(\beta,B) \) is as defined in \eqref{eq:thm:thermodynamic-limit:pressure:explicit}.
\end{Proposition}
Lastly, in the third step, we take the limit \(B\to\infty\), for which we show that 
\eqn{\label{eq:proof:outline:6}
\varphi(\beta,B)-B\to \frac{\E[D(\emp)]}{2}\beta,
}
and \(\psi_n(\beta,B)-B\) can be uniformly-well approximated by \({\E[D(\emp)]}\beta/2\). This limiting argument requires some delicate computation.

Note that, Proposition~\ref{prop:mag:lim} and Lemma~\ref{lem:h:magnetization} in fact prove Theorem~\ref{thm:magnetization} (a). Now, we are ready to complete the proof of Theorem~\ref{thm:thermodynamic-limit:pressure}, subject to Proposition~\ref{prop:mag:lim} and \ref{prop:explicit:pressure}, and then move on to prove Proposition~\ref{prop:mag:lim} and \ref{prop:explicit:pressure}.

\begin{proof}[Proof of Theorem~\ref{thm:thermodynamic-limit:pressure} subject to Proposition~\ref{prop:mag:lim} and \ref{prop:explicit:pressure}]
	%We first prove the convergence part of the proof, and the explicit form of \(\varphi(\beta,B)\) is proved in Section~\ref{sec:explicit}.
	Since \(\pdv{}{B}\psi_n(\beta,B)\) is positive and bounded above by \(1\), we apply Proposition~\ref{prop:explicit:pressure}, and dominated convergence theorem to \eqref{eq:proof:outline:1} to obtain, for any \(B_1>B_0\),
	\eqn{\label{eq:proof:pressure:1}
		\lim\limits_{n\to\infty} \psi_n(\beta,B_1)-\psi_n(\beta,B_0) = \varphi(\beta,B_1)-\varphi(\beta,B_0)~.
	} 
	Now, we would like to take \(B_1\to\infty\). Note that, by \eqref{eq:thm:thermodynamic-limit:pressure:explicit},
	\eqan{\label{eq:proof:pressure:2}
		&\varphi(\beta,B)-B\\
		&\hspace{0.5cm}= \frac{\E[D(\emp)]}{2}\log \cosh(\beta)-\frac{{\E[D(\emp)]}}{2}\E\Big[\log\big\{1+\tanh(\beta) \tanh(\h(\bff(\hat{\emp} 1)))\tanh(\h(\bff(\hat{\emp}))) \big\}\Big]\nn\\
		&\hspace{4cm}+\E\Big[ \log\Big( \prod\limits_{i=1}^{D(\emp)}\big\{1+\tanh(\beta)\tanh(\h(\bff(\emp i)))\big\}\nn\\
		&\hspace{6cm}+\e^{-2B}\prod\limits_{i=1}^{D(\emp)}\big\{1-\tanh(\beta)\tanh(\h(\bff(\emp i)))\big\} \Big) \Big]~.\nn	
	}
	From the definition of the P\'{o}lya point tree, we obtain that \(\E[D(\emp)]=2m\). Thus, as \(B\to\infty\),
	\eqan{\label{eq:proof:pressure:3}
		&\varphi(\beta,B)-B\\
		&\hspace{0.5cm}=~ m\log \cosh(\beta)-m\E\Big[\log\big\{1+\tanh(\beta) \tanh(\h(\bff(\hat{\emp} 1)))\tanh(\h(\bff(\hat{\emp}))) \big\}\Big]\nn\\
		&\hspace{4cm}+\E\Big[ \log\Big( \prod\limits_{i=1}^{2m}\big\{1+\tanh(\beta)\tanh(\h(\bff(\emp i)))\big\} \Big) \Big]+o_{\sss \prob}(1)~.\nn	
	}
	Since \(\{\h(\omega):\omega\in\Scal\}\) satisfies the fixed-point equation in \eqref{eq:test:1}, these random variables diverge to \(\infty\) as \(B\to\infty\). Therefore, \eqref{eq:proof:pressure:3} can be further simplified as
	\eqn{\label{eq:proof:pressure:4}
		\varphi(\beta,B)-B=m\log \cosh(\beta) + m\log(1+\tanh(\beta))+o_{\prob}(1)=m\beta+o_{\sss \prob}(1)~,
	}
	where the last equality follows from the fact that \(\cosh(\beta)(1+\tanh(\beta))=\e^\beta\). Thus, choosing \(B_2\) arbitrarily large, we can make \(\sup_{B>B_2}\mid \varphi(\beta,B)-B-m\beta \mid\) arbitrarily small.
	
	Now, we prove a similar estimate for \(\psi_n(\beta,B)\). Note that, by considering 
	\((\sigma_v)_{v\in[n]}=\{1\}_{v\in[n]}\),
	\eqn{\label{eq:proof:pressure:5}
		Z_n(\beta,B)\geq \e^{nB+|E(G_n)|\beta}~,
	}
	so that
	\eqn{\label{eq:proof:pressure:6}
		\psi_n(\beta,B)-B\geq \frac{|E(G_n)|}{n}\beta \to m\beta ~.
	}
	On the other hand,
	\eqn{\label{eq:proof:pressure:7}
		\sum\limits_{\{u,v\}\in E(G_n)} \sigma_u\sigma_v \leq |E(G_n)|~.
	}
	Thus, we can upper bound \(Z_n(\beta,B)\) as
	\eqn{\label{eq:proof:pressure:8}
		Z_n(\beta,B)\leq \e^{nB+|E(G_n)|\beta}\sum\limits_{\sigma\in\{-1,1\}^n} \prod\limits_{u\in [n]} \e^{-2B\one_{\{\sigma_u=-1\}}}=\e^{nB+|E(G_n)|\beta}(1+\e^{-2B})^n~.
	}
	Therefore, from \eqref{eq:proof:pressure:6} and \eqref{eq:proof:pressure:8}, we conclude that 
	\eqn{\label{eq:proof:pressure:9}
		0\leq \psi_n(\beta,B)\leq-B-m\beta \log(1+\e^{-2B})~.
	}
	Thus, choosing \(B_3\) sufficiently large, we can also make \(\sup_{B>B_3}\mid \psi_n(\beta,B)-B-m\beta \mid\) arbitrarily small. Choosing \(B_1=\max\{B_2,B_3\}\)
	\eqan{\label{eq:proof:pressure:10}
		\psi_n(\beta,B)=&\varphi(\beta,B)-\varphi(\beta,B_1)+\psi_n(\beta,B)+o(1)\nn\\
		=&\varphi(\beta,B)-\big[ \varphi(\beta,B_1)-B_1-m\beta \big]+\big[ \psi_n(\beta,B)-B_1-m\beta \big]+o(1)~,
	}
	and then letting \(B_1\to \infty\), yields
	\eqn{\label{eq:proof:pressure:11}
		\psi_n(\beta,B)\overset{\prob}{\longrightarrow}\varphi(\beta,B)~,
	}
	completing the proof.
\end{proof}

Now, we move on to prove Proposition~\ref{prop:mag:lim}. The convergence follows immediately from the local convergence of preferential attachment models to the P\'{o}lya point tree, whereas the equality in \eqref{eq:prop:mag:lim} follows from \cite{DMS13} for all locally tree-like graphs.

\begin{proof}[Proof of Proposition~\ref{prop:mag:lim}]
	Now, we remain to prove Proposition~\ref{prop:mag:lim}. By \cite[Theorem~1.9 and Theorem~1.16]{DMS13},
	\eqn{\label{for:prop:new:1}
		\lim\limits_{n\to\infty} \big[ \psi_n(\beta,B_1)- \psi_n(\beta,B_0) \big] = \bar\varphi(\beta,B_1)- \bar\varphi(\beta,B_0)~,
	}
	where by \cite[(1.6)-(1.8)]{DMS13}, \(\bar\varphi(\beta,B)\) is defined as
	\eqan{\label{for:prop:new:2}
		\bar\varphi(\beta,B) = &\ \mathbb{E}_{\mu}\left[ \log\left\{ \sum\limits_{\sigma\in\{-1,1\}} \exp\left(B\sigma\right) \prod\limits_{j=1}^{D(\emp)} \left( \sum\limits_{\sigma_j\in\{-1,1\}} \exp\left(\beta\sigma\sigma_j\right) \nu_{\emp j\to\emp}(\sigma_j) \right) \right\}\right.\nonumber \\
		&\hspace{1.75cm}\left.-\frac{1}{2}\sum\limits_{j=1}^{D(\emp)} \log\left\{ \sum\limits_{\sigma,\sigma_j\in\{-1,1\}} \exp\left(\beta\sigma\sigma_j\right) \nu_{\emp j\to\emp}(\sigma_j) \nu_{\emp\to\emp j}(\sigma) \right\} \right]~.}
	
	By \cite[Lemma~2.1]{DMS13}, \(\psi_{n}(\beta,B)\) is uniformly bounded and equicontinuous sequence of functions. Furthermore, \(\pdv{}{B}\psi_{n}(\beta,B)\) is uniformly bounded. Now, by the dominated convergence theorem and the fundamental theorem of calculus, 
	%Now, for any \(B>0\), choosing \(B_1=B+h\) and \(B_0=B\), and taking the limit as \(h\) to \(0\), we obtain
	\eqn{\label{for:prop:new:3}
	\frac{\partial}{\partial B}\psi_n(\beta,B)\overset{\prob}{\to}\frac{\partial}{\partial B}\bar\varphi(\beta,B)~.
	}
	
	Since for any vertex \(v\in[n],~\anglbrkt{\sigma_v}_{B_v(\ell)}^{+/f}\) is a bounded function that only depends on the $\ell$-neighbourhood, by local convergence of preferential attachment models to the P\'{o}lya point tree, we obtain
	\eqref{eq:proof:outline:4}. Next, by \cite[Lemma~3.1]{DGvdH10}, both these bounds converge to the same limit, as \(\ell\to\infty\). Therefore,
	\eqn{\label{for:prop:new:4}
		\frac{\partial}{\partial B}\bar\varphi(\beta,B) = \E \big[ \anglbrkt{\sigma_\emp}_\mu \big]
	}
	completing the proof.
\end{proof}

%\subsection{Explicit expression}\label{sec:explicit}

In Proposition~\ref{prop:explicit:pressure}, we prove the equivalence of the expressions in the RHS of \eqref{eq:thm:thermodynamic-limit:pressure:explicit} and \eqref{for:prop:new:2}.
To prove the equivalence of these two expressions for the thermodynamic limit of the pressure, we must use two distributional properties of the nodes of the P\'{o}lya point tree. 
The following lemma proves a distributional equivalence related to $\Old$-labelled children of the root in the P\'{o}lya point tree:

\medskip
\begin{lemma}[Distributional equivalence of $\Old$-labelled children]
	\label{lem:old:distributional:equivalence}
	Let $\bff(\hat{\emp})=(U,\Young)$, where $U\sim\Unif[0,1]$. Then, for any $i\in[m]$,
	\eqn{\label{eq:lem:old:distributional:equivalence}
		\big( \h_{-i}(\bff(\emp)),\h(\bff(\emp i)) \big)\overset{d}{=}\big( \h(\bff(\hat{\emp})),\h(\bff(\hat{\emp} 1)) \big)~.}
\end{lemma}

The next lemma proves a similar distributional equivalence for $\Young$-labelled children of the root:

\medskip
\begin{lemma}[Distributional equivalence of $\Young$-labelled children]\label{lem:young:distributional:equivalence}
	Let $\tilde{\emp}$ have label $\Old$ and an age distributed according to a density, $a\mapsto\gamma(a)$, where $\gamma(a)$ is given by
	\eqn{\label{eq:def:gamma(a)} \gamma(a)=\frac{m+\delta}{m}(a^{\chi-1}-1)\quad\mbox{for }a\in[0,1]~.}
	Conditionally on $\tilde{\emp}$ having age $a$, let $\tilde{\emp}\tilde{1}$ have an age distributed according to $f_a$ and label $\Young$, where $f_a(x)$ is defined as
	\eqn{\label{eq:def:f(a)}
		f_a(x)=(1-\chi)\frac{x^{-\chi}}{1-a^{1-\chi}}\one_{\{x\geq a\}}~.}
	Then,
	\eqn{\label{eq:lem:young:distributional:equivalence}
		\big( \h(\bff(\tilde{\emp}\tilde{1})),\h(\bff(\tilde{\emp})) \big)\overset{d}{=}\big( \h(\bff(\hat{\emp})),\h(\bff(\hat{\emp} 1)) \big)~.
	}
\end{lemma}

The following lemma establishes a convenient size-biasing argument related to the in-degree of the root:

\medskip
\begin{lemma}[Size-biased mixed Poisson]\label{lem:size-biased:poisson}
	Let $\tilde{\emp}$ have label $\Old$ and age having the density, $a\mapsto\gamma(a)$ defined in \eqref{eq:def:gamma(a)}. Then, the following size-biasing result regarding the in-degree of the root holds:
	\eqn{\label{eq:lem:size-biased:poisson}
		\ell\prob\Big( d_\emp^{\rm(in)}=\ell\mid A_\emp=a \Big) = \E\big[ d_\emp^{\rm(in)} \big]\prob\Big( d_{\tilde{\emp}}^{\rm(in)}=\ell-1\mid \tilde{\emp}=(a,\Old) \Big)\gamma(a)~.}
\end{lemma}

We will first prove Proposition~\ref{prop:explicit:pressure} using these lemmas and then proceed to prove the lemmas themselves.

\begin{proof}[Proof of Proposition~\ref{prop:explicit:pressure}]
	%The proof of the convergence of the pressure to its thermodynamic limit, as shown in \cite[Theorem~1.4]{DGvdH10}, does not rely on any specific properties of the configuration model, except for the condition $|E_n|/n \leq c$ for some constant $c \in \R$ in \cite[(1.26)]{DGvdH10}, where $E_n$ is the edge set of the configuration model of size $n$. This condition also holds true for preferential attachment models. Since the remainder of the proof is independent of the particular random graph model being considered, it follows for preferential attachment models as well. Alternatively, this part of the proof also follows from \cite[Theorem~1.9]{DMS13}.
	
	We now proceed to prove the explicit expression for the thermodynamic limit of the pressure. By \eqref{for:prop:new:2},
	\begin{align}\label{for:thm:pressure-limit:1}
		\bar\varphi(\beta,B) = &\ \mathbb{E}_{\mu}\left[ \log\left\{ \sum\limits_{\sigma\in\{-1,1\}} \exp\left(B\sigma\right) \prod\limits_{j=1}^{D(\emp)} \left( \sum\limits_{\sigma_j\in\{-1,1\}} \exp\left(\beta\sigma\sigma_j\right) \nu_{\emp j\to\emp}(\sigma_j) \right) \right\}\right.\nonumber \\
		&\hspace{1.75cm}\left.-\frac{1}{2}\sum\limits_{j=1}^{D(\emp)} \log\left\{ \sum\limits_{\sigma,\sigma_j\in\{-1,1\}} \exp\left(\beta\sigma\sigma_j\right) \nu_{\emp j\to\emp}(\sigma_j) \nu_{\emp\to\emp j}(\sigma) \right\} \right]~.
	\end{align}
	To derive an expression similar to \eqref{eq:thm:thermodynamic-limit:pressure:explicit} from \eqref{for:thm:pressure-limit:1}, we simplify the terms on the RHS of \eqref{for:thm:pressure-limit:1} step by step.
	We start by writing the first term as
	\eqan{\label{for:thm:pressure-limit:02}
		%&\E_\mu\left[ \log\Big\{ \sum\limits_{\sigma\in\{-1,1\}}\e^{B\sigma}\prod\limits_{j=1}^{D(\emp)}\Big( \sum\limits_{\sigma_j\in\{-1,1\}}\e^{\beta\sigma\sigma_j}\nu_{\emp j\to\emp}(\sigma_j) \Big) \Big\} \right]\nn\\
		%=
		&\E_\mu\left[ \log\Big\{ \e^B\prod\limits_{j=1}^{D(\emp)}\Big( \e^\beta\nu_{\emp j\to\emp}(+1)+\e^{-\beta}\nu_{\emp j\to\emp}(-1) \Big)\right.\nn\\
		&\hspace{5cm}\left.+\e^{-B}\prod\limits_{j=1}^{D(\emp)}\Big( \e^{-\beta}\nu_{\emp j\to\emp}(+1)+\e^{\beta}\nu_{\emp j\to\emp}(-1) \Big) \Big\} \right]\nn\\
		=&\E_\mu\left[ \log\Bigg\{ \e^B\prod\limits_{j=1}^{D(\emp)}\frac{2\Big( \e^\beta\nu_{\emp j\to\emp}(+1)+\e^{-\beta}\nu_{\emp j\to\emp}(-1) \Big)}{\e^{\beta}+\e^{-\beta}}\right.\\
		&\hspace{1.5cm}\left.+\e^{-B}\prod\limits_{j=1}^{D(\emp)}\frac{2\Big( \e^{-\beta}\nu_{\emp j\to\emp}(+1)+\e^{\beta}\nu_{\emp j\to\emp}(-1) \Big)}{\e^{\beta}+\e^{-\beta}} \Bigg\}+D(\emp)\log\Bigg\{ \frac{\e^{\beta}+\e^{-\beta}}{2} \Bigg\} \right]~.\nn
	}
	Using that $\big(2\big( \e^\beta\nu_{\emp j\to\emp}(+1)+\e^{-\beta}\nu_{\emp j\to\emp}(-1) \big)\big)/\big(\e^{\beta}+\e^{-\beta}\big)$ equals $\big( 1+\tanh(\beta)\big(2\nu_{\emp j\to\emp}(+1)-1 \big) \big)$, we arrive at
	\eqan{\label{for:thm:pressure-limit:2}
		&\E_\mu[D(\emp)]\log\cosh(\beta)+\E\left[ \log\Bigg\{ \e^B\prod\limits_{j=1}^{D(\emp)}\Big( 1+\tanh(\beta)\big(2\nu_{\emp j\to\emp}(+1)-1 \big) \Big)\right.\nn\\
		&\hspace{5cm}\left. +\e^{-B}\prod\limits_{j=1}^{D(\emp)}\Big( 1-\tanh(\beta)\big(2\nu_{\emp j\to\emp}(+1)-1 \big) \Big) \Bigg\} \right]~.
	}
	Note that $\magn{\emp j\to}=2\nu_{\emp j\to\emp}(+1)-1$ by the definition of the magnetization. Therefore, using Lemma~\ref{lem:h:magnetization} and \( \h(\bff(\emp j))=\arctanh{\big( \magn{\emp j \to} \big)} \), we simplify \eqref{for:thm:pressure-limit:2} as
	\eqan{\label{for:thm:pressure-limit:3}
		%&\E_\mu\left[ \log\Big\{ \sum\limits_{\sigma\in\{-1,1\}}\e^{B\sigma}\prod\limits_{j=1}^{D(\emp)}\Big( \sum\limits_{\sigma_j\in\{-1,1\}}\e^{\beta\sigma\sigma_j}\nu_{\emp j\to\emp}(\sigma_j) \Big) \Big\} \right]\nn\\
		&\E_\mu[D(\emp)]\log\cosh(\beta)+\E\Bigg[ \log\Big\{ \e^B\prod\limits_{j=1}^{D(\emp)}\Big( 1+\tanh(\beta)\tanh\big(\h(\bff(\emp j))\big) \Big)\nn\\
		&\hspace{5.5cm} +\e^{-B}\prod\limits_{j=1}^{D(\emp)}\Big( 1-\tanh(\beta)\tanh\big(\h(\bff(\emp j))\big) \Big) \Big\} \Bigg]~.
	}
	Similarly, we simplify the second term on the RHS of \eqref{for:thm:pressure-limit:1} as
	\eqan{\label{for:thm:pressure-limit:old}
		&\E_\mu\Bigg[ \sum\limits_{j=1}^{D(\emp)}\log\Big\{ \sum\limits_{\sigma,\sigma_j\in\{-1,1\}}\e^{\beta\sigma\sigma_j}\nu_{\emp j\to\emp}(\sigma_j)\nu_{\emp\to\emp j}(\sigma) \Big\} \Bigg]\nn\\
		&\hspace{0.5cm}=\E_{\mu}\Bigg[ \sum\limits_{j=1}^{D(\emp)} \log\Bigg\{ \e^{\beta}\nu_{\emp j\to\emp}(+1)\nu_{\emp \to\emp j}(+1)+\e^{-\beta}\nu_{\emp j\to\emp}(-1)\nu_{\emp \to\emp j}(+1)\nn\\
		&\hspace{2.65cm}+\e^{-\beta}\nu_{\emp j\to\emp}(+1)\nu_{\emp \to\emp j}(-1)+\e^{\beta}\nu_{\emp j\to\emp}(-1)\nu_{\emp \to\emp j}(-1) \Bigg\} \Bigg]~.
%		=&\E_{\mu}\Bigg[ \sum\limits_{j=1}^{D(\emp)} \log\Bigg( \frac{2}{\e^{\beta}+\e^{-\beta}}\Big\{ \e^{\beta}\nu_{\emp j\to\emp}(+1)\nu_{\emp \to\emp j}(+1)+\e^{-\beta}\nu_{\emp j\to\emp}(-1)\nu_{\emp \to\emp j}(+1)\nn\\
%		&\hspace{2.5cm}+\e^{-\beta}\nu_{\emp j\to\emp}(+1)\nu_{\emp \to\emp j}(-1)+\e^{\beta}\nu_{\emp j\to\emp}(-1)\nu_{\emp \to\emp j}(-1) \Big\}\Bigg) \Bigg]\nn\\
%		&\hspace{7cm}+\E_\mu[D(\emp)]\log\Big\{ \frac{\e^{\beta}+\e^{-\beta}}{2} \Big\}\nn\\
%		=&\E_{\mu}\Bigg[ \sum\limits_{j=1}^{D(\emp)}\log\Big\{ 1+\tanh(\beta)\big( 2\nu_{\emp j\to\emp}(+1)-1 \big)\big( 2\nu_{\emp\to\emp j}(+1)-1 \big) \Big\} \Bigg]\nn\\
%		&\hspace{8cm}+\E_{\mu}[D(\emp)]\log\cosh(\beta)~.
	}
	Next, note that for all \(x,y\in\R\),
	\[
		\begin{aligned}
			&2\big\{\e^{\beta}xy+\e^{-\beta}x(1-y)+\e^{-\beta}(1-x)y+\e^{\beta}(1-x)(1-y)\big\}\nn\\
		&\hspace{4cm}=  (\e^\beta+\e^{-\beta})+(\e^\beta-\e^{-\beta})(2x-1)(2y-1)~,
		\end{aligned}
	\]
	to arrive at
	\eqan{\label{for:thm:pressure-limit:4}
		&\E_\mu\Bigg[ \sum\limits_{j=1}^{D(\emp)}\log\Big\{ \sum\limits_{\sigma,\sigma_j\in\{-1,1\}}\e^{\beta\sigma\sigma_j}\nu_{\emp j\to\emp}(\sigma_j)\nu_{\emp\to\emp j}(\sigma) \Big\} \Bigg]\nn\\
		&\hspace{1cm}=~\E_{\mu}\Bigg[ \sum\limits_{j=1}^{D(\emp)}\log\Big\{ 1+\tanh(\beta)\big( 2\nu_{\emp j\to\emp}(+1)-1 \big)\big( 2\nu_{\emp\to\emp j}(+1)-1 \big) \Big\} \Bigg]\\
		&\hspace{1.5cm}+\E_{\mu}[D(\emp)]\log\cosh(\beta)~.\nn
	}
	Again using definition of $\magn{\to \emp j}$ and Lemma~\ref{lem:h:magnetization}, we can rewrite the LHS of \eqref{for:thm:pressure-limit:4} as
	\eqn{\label{for:thm:pressure-limit:5}
		 \E_{\mu}\Bigg[ \sum\limits_{j=1}^{D(\emp)}\log\Big\{ 1+\tanh(\beta)\tanh\big(\h_{-j}(\bff(\emp))\big)\tanh\big(\h(\bff(\emp j))\big) \Big\} \Bigg]+\E_\mu[D(\emp)]\log\cosh(\beta)~.
	}
	Now, plugging in the simplified forms obtained in \eqref{for:thm:pressure-limit:3} and \eqref{for:thm:pressure-limit:5} in \eqref{for:thm:pressure-limit:1}, we obtain
	\eqan{\label{for:thm:pressure-limit:6}
		&\bar\varphi(\beta,B)\\
		=& \frac{\E_\mu[D(\emp)]}{2}\log\cosh(\beta)-\frac{1}{2}\E_{\mu}\Bigg[ \sum\limits_{j=1}^{D(\emp)}\log\Big\{ 1+\tanh(\beta)\tanh\big(\h_{-j}(\bff(\emp))\big)\tanh\big(\h(\bff(\emp j))\big) \Big\} \Bigg]\nn\\
		&\hspace{1cm}+\E_\mu\Bigg[ \log\Big\{ \e^B\prod\limits_{j=1}^{D(\emp)}\Big( 1+\tanh(\beta)\tanh\big(\h(\bff(\emp j))\big) \Big)\nn\\
		&\hspace{5cm}+\e^{-B}\prod\limits_{j=1}^{D(\emp)}\Big( 1-\tanh(\beta)\tanh\big(\h(\bff(\emp j))\big) \Big) \Big\} \Bigg]~.\nn
	}
	Next, we use the distributional properties of the nodes in the P\'{o}lya point tree, as outlined in Lemmas~\ref{lem:old:distributional:equivalence} and \ref{lem:young:distributional:equivalence}, to conclude that the RHS of \eqref{for:thm:pressure-limit:6} further simplifies to the RHS of \eqref{eq:thm:thermodynamic-limit:pressure:explicit}. Note that $D(\emp)=m+d_\emp^{\rm(in)}$, where $d_\emp^{\rm(in)}$ is the number of $\Young$-labelled children of the root in the P\'{o}lya point tree. Therefore, the second term on the RHS of \eqref{for:thm:pressure-limit:6} can be divided into two parts: one comprising the contribution from the $\Old$-labelled children of the root, and the other comprising the contribution from the $\Young$-labelled children of the root, i.e.,
	\eqan{\label{for:thm:pressure-limit:7}
		&\E_{\mu}\Bigg[ \sum\limits_{j=1}^{D(\emp)}\log\Big\{ 1+\tanh(\beta)\tanh\big(\h_{-j}(\bff(\emp))\big)\tanh\big(\h(\bff(\emp j))\big) \Big\} \Bigg]\\
		=&\E_{\mu}\Bigg[ \sum\limits_{j=1}^{m}\log\Big\{ 1+\tanh(\beta)\tanh\big(\h_{-j}(\bff(\emp))\big)\tanh\big(\h(\bff(\emp j))\big) \Big\} \Bigg]\nn\\
		&\hspace{3cm}+\E_{\mu}\Bigg[ \sum\limits_{j=m+1}^{m+d_\emp^{\rm(in)}}\log\Big\{ 1+\tanh(\beta)\tanh\big(\h_{-j}(\bff(\emp))\big)\tanh\big(\h(\bff(\emp j))\big) \Big\} \Bigg]~.\nn
	}
	Using Lemma~\ref{lem:old:distributional:equivalence}, we simplify the first sum in the RHS of \eqref{for:thm:pressure-limit:7} as
	\eqn{\label{for:thm:pressure-limit:8}
		m\E_{\mu}\Big[ \log\Big\{ 1+\tanh(\beta)\tanh\big(\h(\bff(\hat{\emp}))\big)\tanh\big(\h(\bff(\hat{\emp} 1))\big) \Big\} \Big]~.}
	Next, we simplify the second term on the RHS of \eqref{for:thm:pressure-limit:7}. From the definition of the P\'{o}lya point tree, we know that $d_\emp^{\rm(in)}$ is a mixed-Poisson random variable with intensity parameter $\Gamma\big(A_\emp^{\chi-1}-1\big)$, where $\Gamma \sim \text{Gamma}(m+\delta, 1)$ and $A_\emp \sim \text{Unif}(0,1)$. Furthermore, conditionally on $\{A_\emp = a\}$ and $d_\emp^{\rm(in)} = \ell$, the ages of $\{\emp(m+1), \ldots, \emp(m+\ell)\}$ are i.i.d.\ with density $f_a$ defined in \eqref{eq:def:f(a)}. Conditionally on $\{A_\emp = a\}$ and $d_{\emp}^{\rm(in)} = \ell$, the age of $\emp\tilde{1}$ has density $f_a$ and label $\Young$.
	Therefore,
	\eqan{\label{for:thm:pressure-limit:9}
		&\E_{\mu}\Bigg[ \sum\limits_{j=m+1}^{m+d_\emp^{\rm(in)}}\log\Big\{ 1+\tanh(\beta)\tanh\big(\h_{-j}(\bff(\emp))\big)\tanh\big(\h(\bff(\emp j))\big) \Big\} \Bigg]\\
		=&\int\limits_0^1\sum\limits_{\ell=1}^\infty \E_{\mu}\Bigg[ \sum\limits_{j=m+1}^{m+\ell}\log\Big\{ 1+\tanh(\beta)\tanh\big(\h_{-j}(\bff(\emp))\big)\tanh\big(\h(\bff(\emp j))\big) \Big\}\mid A_\emp=a,d_\emp^{\rm(in)}=\ell \Bigg]\nn\\
		&\hspace{9.5cm}\times\prob\big( d_\emp^{\rm(in)}=\ell\mid A_\emp=a \big)\,da\nn\\
		=&\int\limits_0^1\sum\limits_{\ell=1}^\infty \E_{\mu}\Bigg[ \log\Big\{ 1+\tanh(\beta)\tanh\big(\h_{-\tilde{1}}(\bff(\emp))\big)\tanh\big(\h(\bff(\emp \tilde{1}))\big) \Big\}\mid A_\emp=a,d_\emp^{\rm(in)}=\ell \Bigg]\nn\\
		&\hspace{9.5cm}\times\ell\prob\big( d_\emp^{\rm(in)}=\ell\mid A_\emp=a \big)\,da~.\nn
	}
	Note that, conditionally on $\{A_\emp = a, d_\emp^{\rm(in)} = \ell\}$, $\h(\emp\tilde{1})$ and $\h_{-\tilde{1}}(\emp)$ are independent. Let $\tilde{\emp}$ have label $\Old$ and age with density $\gamma(a)$. Conditionally on $A_{\tilde{\emp}} = a$, $\tilde{\emp}\tilde{1}$ has label $\Young$ and its age has density $f_a$. %From the definitions in \eqref{eq:distributional-recursion:1}
	By Lemma~\ref{lem:h:magnetization},
	\begin{equation}
		\label{for:thm:pressure-limit:10}
		\begin{aligned}
			\h_{-\tilde{1}}(\bff(\emp)) \mid \{A_\emp = a, d_\emp^{\rm(in)} = \ell\} ~ &\overset{d}{=} ~ \h(\bff(\tilde{\emp})) \mid \{\tilde{\emp} = (a, \Old), d_{\tilde{\emp}}^{\rm(in)} = \ell - 1\}, \\
			\h(\bff(\emp\tilde{1})) \mid \{A_\emp = a, d_\emp^{\rm(in)} = \ell\} ~ &\overset{d}{=} ~ \h(\bff(\tilde{\emp}\tilde{1})) \mid \{\tilde{\emp} = (a, \Old), d_{\tilde{\emp}}^{\rm(in)} = \ell - 1\}.
		\end{aligned}
	\end{equation}
	
	Conditionally on $\tilde{\emp} = (a, \Old)$, drawing $\h(\tilde{\emp})$ and $\h(\tilde{\emp}\tilde{1})$ independently leads to the distributional equality
	\begin{equation}
		\label{for:thm:pressure-limit:11}
		\big(\h(\bff(\tilde{\emp})), \h(\bff(\tilde{\emp}\tilde{1}))\big) \mid \{\tilde{\emp} = (a, \Old), d_{\tilde{\emp}}^{\rm(in)} = \ell - 1\}  \overset{d}{=}  \big(\h_{-\tilde{1}}(\bff(\emp)), \h(\bff(\emp\tilde{1}))\big) \mid \{A_\emp = a, d_\emp^{\rm(in)} = \ell\}.
	\end{equation}
	
	Hence, using \eqref{for:thm:pressure-limit:11} and Lemma~\ref{lem:size-biased:poisson}, \eqref{for:thm:pressure-limit:10} can be simplified as
	\eqan{\label{for:thm:pressure-limit:12}
		\E_{\mu}&\Bigg[ \sum\limits_{j=m+1}^{m+d_\emp^{\rm(in)}}\log\Big\{ 1+\tanh(\beta)\tanh\big(\h_{-j}(\bff(\emp))\big)\tanh\big(\h(\bff(\emp j))\big) \Big\} \Bigg]\nn\\
		=&\int\limits_0^1\sum\limits_{\ell=1}^\infty \E_{\mu}\Bigg[ \log\Big\{ 1+\tanh(\beta)\tanh\big(\h(\bff(\tilde{\emp}\tilde{1}))\big)\tanh\big(\h( \bff(\tilde{\emp}))\big) \Big\}\mid \tilde{\emp}=(a,\Old),d_{\tilde{\emp}}^{\rm(in)}=\ell-1 \Bigg]\nn\\
		&\hspace{6.5cm}\times\E\big[ d_{\emp}^{\rm(in)} \big]\prob\big( d_{\tilde{\emp}}^{\rm(in)}=\ell-1\mid \tilde{\emp}=(a,\Old) \big)\gamma(a)\,da\nn\\
		=&~\E\big[ d_{\emp}^{\rm(in)} \big]\E_{\mu}\Big[ \log\Big\{ 1+\tanh(\beta)\tanh\big(\h(\bff(\tilde{\emp}\tilde{1}))\big)\tanh\big(\h(\bff( \tilde{\emp}))\big) \Big\} \Big]~.
	}
	By Lemma~\ref{lem:young:distributional:equivalence} and \eqref{for:thm:pressure-limit:7},~\eqref{for:thm:pressure-limit:8} and \eqref{for:thm:pressure-limit:12}, we conclude that
	\eqan{\label{for:thm:pressure-limit:13}
		&\E_{\mu}\Bigg[ \sum\limits_{j=1}^{D(\emp)}\log\Big\{ 1+\tanh(\beta)\tanh\big(\h_{-j}(\bff(\emp))\big)\tanh\big(\h(\bff(\emp j))\big) \Big\} \Bigg]\nn\\
		&\hspace{1cm}=~\Big(m+\E_{\mu}\big[ d_{\emp}^{\rm(in)} \big]\Big)\E_{\mu}\Big[ \log\Big\{ 1+\tanh(\beta)\tanh\big(\h(\bff({\emp}{1}))\big)\tanh\big(\h(\bff( \hat{\emp}))\big) \Big\} \Big]\nn\\
		&\hspace{1cm}=~\E_\mu\big[ D(\emp) \big]\E_{\mu}\Big[ \log\Big\{ 1+\tanh(\beta)\tanh\big(\h(\bff({\emp}{1}))\big)\tanh\big(\h( \bff(\hat{\emp}))\big) \Big\} \Big]~.}
	Therefore, substituting \eqref{for:thm:pressure-limit:13} in \eqref{for:thm:pressure-limit:6}, we obtain that for \(\bar\varphi(\beta,B)\) equals the expression in \eqref{eq:thm:thermodynamic-limit:pressure:explicit} for $\varphi(\beta,B)$.
\end{proof}

Now that we have proved Proposition~\ref{prop:explicit:pressure} subject to Lemmas~\ref{lem:old:distributional:equivalence}, ~\ref{lem:young:distributional:equivalence} and \ref{lem:size-biased:poisson}, we provide the proof to these lemmata one by one. First, we prove Lemma~\ref{lem:old:distributional:equivalence}:
\begin{proof}[Proof of Lemma~\ref{lem:old:distributional:equivalence}]
	Note that in the P\'{o}lya point tree, the ages of the $\Old$-labelled children are an exchangeable sequence of random variables, i.e., for any permutation $\Upsilon:\N\to\N$,
	\eqn{\label{for:lem:old:distributional:equivalence:1}
		\big( A_{\emp 1},\ldots,A_{\emp m} \big)\overset{d}{=}\big( A_{\emp \Upsilon(1)},\ldots,A_{\emp \Upsilon(m)} \big)~.}
	Fix any $i\in[m]$, and define for \(x\in[m]\) 
	\[
	\Upsilon_i(x)=\begin{cases}
		i \quad \mbox{if }x=1,\\
		1 \quad \mbox{if }x=i,\\
		x \quad \mbox{otherwise.} 
	\end{cases}
	\]
	Then, by the distributional recursion in Lemma~\ref{lem:h:magnetization}, we obtain
	\eqn{\label{for:lem:old:distributional:equivalence:2}
		\big(\h(\bff(\emp i)),\h_{-i}(\bff(\emp))\big)\overset{d}{=}\Big( \h(\bff(\emp \Upsilon_i(1))), B+\sum\limits_{j=2}^{D(\emp)}\arctanh\big\{\tanh(\beta)\tanh\big(\h(\bff(\emp \Upsilon_i(j)))\big)\big\} \Big)~.}
	Further, for any $\omega\in\Scal$, the distribution of $h(\omega)$ is dependent only on $\omega$. Hence using \eqref{for:lem:old:distributional:equivalence:1},
	\eqan{\label{for:lem:old:distributional:equivalence:3}
		&\Big( \h(\bff(\emp \Upsilon_i(1))), B+\sum\limits_{j=2}^{D(\emp)}\arctanh\big\{\tanh(\beta)\tanh\big(\h(\bff(\emp \Upsilon_i(j)))\big)\big\} \Big)\nn\\
		&\hspace{2cm}\overset{d}{=}\Big( \h(\bff(\emp 1)), B+\sum\limits_{j=2}^{D(\emp)}\arctanh\big\{\tanh(\beta)\tanh\big(\h(\bff(\emp j))\big)\big\} \Big)~.}
	Note that $\hat{\emp}= (U,\Young)$ is labelled $\Young$, and by the definition of the P\'{o}lya point tree, we obtain that the root and $\Young$-labelled nodes of the P\'{o}lya point tree are i.i.d. Furthermore, the number of $\Old$-labelled children of the root is one more than the number of $\Young$-labelled children in the P\'{o}lya point tree. Therefore, with $\hat{\emp}$ having label $\Young$ and age $A_\emp \sim \text{Unif}[0,1]$, we obtain
	\begin{equation}
		\label{for:lem:old:distributional:equivalence:4}
		B + \sum_{j=2}^{D(\emp)} \arctanh\left\{ \tanh(\beta) \tanh\left( \h(\bff(\emp j)) \right) \right\} \overset{d}{=} \h(\bff(\hat{\emp})).
	\end{equation}
	Furthermore, conditionally on $\{A_\emp = a\}$, $\h(\bff(\emp1))$ and $\{\h(\bff(\emp2)), \ldots, \h(\bff(\emp D(\emp)))\}$ are independent, and $\h(\bff(\emp 1))$ is equal in distribution to $\h(\bff(\hat{\emp}1))$. Therefore, conditionally on $\{A_\emp = a\}$, drawing $\h(\bff(\hat{\emp}))$ and $\h(\bff(\hat{\emp}1))$ independently leads to
	\begin{equation}
		\label{for:lem:old:distributional:equivalence:5}
		\Big( \h(\bff(\emp1)), B + \sum_{j=2}^{D(\emp)} \arctanh\left\{ \tanh(\beta) \tanh\left( \h(\bff(\emp j)) \right) \right\} \Big) \overset{d}{=} \left( \h(\bff(\hat{\emp}1)), \h(\bff(\hat{\emp})) \right).
	\end{equation}
	Hence, Lemma~\ref{lem:old:distributional:equivalence} follows immediately from \eqref{for:lem:old:distributional:equivalence:5}.
\end{proof}
Next, we prove Lemma~\ref{lem:young:distributional:equivalence}:
\begin{proof}[Proof of Lemma~\ref{lem:young:distributional:equivalence}]
	To prove this lemma, we show that the joint distributions of the ages of $\big( \tilde{\emp}\tilde{1},\tilde{\emp} \big)$ and $\big( \hat{\emp},\hat{\emp}1 \big)$ are equal. First, we compute the joint distribution of the ages of $\big(\hat{\emp},\hat{\emp}1\big)$ as
	\eqn{\label{for:lem:distributional:equivalence:1}
		\prob\big( A_{\hat{\emp}}\leq x,A_{\hat{\emp}{1}}\leq y \big)=~\int\limits_0^x\prob\big( A_{\hat{\emp}{1}}\leq y\mid A_{\hat{\emp}}=t \big)\,dt =	\int\limits_0^x	\prob\big( U^{1/\chi}t\leq y\mid A_{\hat{\emp}}=t \big)\,dt~.
	}
	Note that $U$ and $A_{\hat{\emp}}$ are independent. Hence, upon further simplification,
	\eqn{\label{for:lem:distributional:equivalence:2}
		\prob\big( A_{\hat{\emp}}\leq x,A_{\hat{\emp}{1}}\leq y \big)=~\begin{cases}
			x &\mbox{if }x<y~,\\
			\frac{1}{1-\chi}y^{\chi}x^{1-\chi}-\frac{\chi}{1-\chi}y\hspace{3pt} &\mbox{if }x\geq y~.
	\end{cases}}
	Next, we compute the joint distribution of the ages of $\big(\tilde{\emp}\tilde{1},\tilde{\emp}\big)$ as
	\eqn{\label{for:lem:distributional:equivalence:3}
		\prob\big(A_{\tilde{\emp}\tilde{1}}\leq x,A_{\tilde{\emp}}\leq y\big)=~\int\limits_0^y \prob\big(A_{\tilde{\emp}\tilde{1}}\leq x\mid A_{\tilde{\emp}}=t \big)\gamma(t)\,dt= \int\limits_0^y \int\limits_0^x f_t(s)\gamma(t)\,ds\,dt~,
	}
	where \(\gamma(\cdot)\)  and \(f_t(\cdot)\) is as defined in \eqref{eq:def:gamma(a)} and \eqref{eq:def:f(a)}, respectively.
	The last equality in \eqref{for:lem:distributional:equivalence:3} is due to the fact that conditionally on $\{\tilde{\emp}=(t,\Old)\},~A_{\tilde{\emp}\tilde{1}}$ has density $f_t(\cdot)$. Therefore, after simplifying \eqref{for:lem:distributional:equivalence:3},
	\eqn{\label{for:lem:distributional:equivalence:4}
		\prob\big(A_{\tilde{\emp}\tilde{1}}\leq x,A_{\tilde{\emp}}\leq y\big)=\begin{cases}
			x &\mbox{if }x<y~,\\
			\frac{1}{1-\chi}y^{\chi}x^{1-\chi}-\frac{\chi}{1-\chi}y\hspace{3pt} &\mbox{if }x\geq y~,
	\end{cases}}
	proving the distributional equivalence of $\big( \tilde{\emp}\tilde{1},\tilde{\emp} \big)$ and $\big( \hat{\emp},\hat{\emp}1 \big)$. Since $\{\h(\omega):\omega\in\Scal\}$ are drawn independently, the joint distributions of $\big( \h(\bff(\tilde{\emp}\tilde{1})),\h(\bff(\tilde{\emp}))\big)$ and $\big( \h(\bff(\hat{\emp})),\h(\bff(\hat{\emp}1)) \big)$ are equal.
\end{proof}
Lastly, we address the proof of Lemma~\ref{lem:size-biased:poisson}. This lemma involves a size-biasing argument for the in-degree of the root of the P\'{o}lya point tree. Since the in-degree of the P\'{o}lya point tree follows a mixed-Poisson distribution, the size-biasing argument introduces an additional term alongside its expectation:

\begin{proof}[Proof of Lemma~\ref{lem:size-biased:poisson}]
	From the definition of the P\'{o}lya point tree, we know that $d_\emp^{\rm(in)}$ is a mixed-Poisson random variable with a mixing distribution $\Gamma_{\emp}\big(A_\emp^{\chi-1}-1\big)$, where $A_\emp$ is the age of the root of the P\'{o}lya point tree. Therefore, conditionally on $A_\emp = a$, $d_\emp^{\rm(in)}$ is a mixed Poisson random variable with a mixing distribution $\Gamma\big(a^{\chi-1}-1\big)$, where $\Gamma$ is a Gamma random variable with parameters $m + \delta$ and $1$. Thus,
	\eqan{\label{for:lem:size-bias:1}
		\ell \prob\big( d_\emp^{\rm(in)}=\ell \mid A_\emp=a \big)&=\ell \prob\big(\Poi\big( \Gamma(a^{\chi-1}-1) \big)=\ell  \big)\nn\\
		&=\int\limits_0^\infty \ell \prob\big(\Poi\big( w(a^{\chi-1}-1) \big)=\ell  \big)w^{m+\delta-1}\frac{\e^{-w}}{\Gamma(m+\delta)}\,dw~,
	}
	where $\Gamma(c)$ is the gamma function evaluated at $c$ for any $c \in \mathbb{R}^+$. From a direct computation of the probability mass function of a Poisson random variable with a fixed parameter $\lambda$, for all $\ell \in \mathbb{N}$,
	
	\begin{equation}
		\label{for:lem:size-bias:2}
		\ell \, \prob(\Poi(\lambda) = \ell) = \lambda \, \prob(\Poi(\lambda) = \ell - 1).
	\end{equation}
	Using \eqref{eq:def:gamma(a)} and\eqref{for:lem:size-bias:2}, the RHS of \eqref{for:lem:size-bias:1} can be simplified to
	\eqan{\label{for:lem:size-bias:3}
		&\int\limits_0^\infty \ell \prob\big(\Poi\big( w(a^{\chi-1}-1) \big)=\ell  \big)w^{m+\delta-1}\frac{\e^{-w}}{\Gamma(m+\delta)}\,dw\nn\\
		=&\int\limits_0^\infty \prob\big(\Poi\big( w(a^{\chi-1}-1) \big)=\ell -1 \big)\frac{m}{m+\delta}\big(a^{\chi-1}-1 \big)w^{m+\delta}\frac{\e^{-w}}{\Gamma(m+\delta)}\,dw~.
	}
	Since $d_\emp^{\rm(in)}$ is a mixed-Poisson random variable with mixing distribution $\Gamma_{\emp}\big(A_\emp^{\chi-1}-1\big)$, 
	\eqan{\label{for:lem:size-bias:4}
		\E\big[d_\emp^{\rm(in)}\big]=\E\big[\Gamma_{\emp}\big(A_\emp^{\chi-1}-1\big)\big]=\E[\Gamma_{\emp}]\E\big[A_\emp^{\chi-1}-1\big]=(m+\delta)(1/\chi-1)=m~.}
	Therefore, using the identity $x\Gamma(x)=\Gamma(x+1)$ and \eqref{for:lem:size-bias:4}, we simplify the RHS of \eqref{for:lem:size-bias:3} further as
	\eqan{\label{for:lem:size-bias:5}
		&\int\limits_0^\infty \prob\big(\Poi\big( w(a^{\chi-1}-1) \big)=\ell -1 \big)\frac{m}{m+\delta}\gamma(a)w^{m+\delta}\frac{\e^{-w}}{\Gamma(m+\delta)}\,dw\nn\\
		=&\E\big[ d_\emp^{\rm(in)} \big]\gamma(a)\int\limits_0^\infty \prob\big(\Poi\big( w(a^{\chi-1}-1) \big)=\ell -1 \big)w^{m+\delta}\frac{\e^{-w}}{\Gamma(m+\delta+1)}\,dw\nn\\
		=& \E\big[ d_\emp^{\rm(in)} \big]\gamma(a) \prob\big(\Poi\big( \Gamma_{\rm{in}}(m+1)(a^{\chi-1}-1) \big)=\ell -1 \big)~,
	}
	where $\Gamma_{\rm{in}}(m+1)$ is a $\rm{Gamma}$ random variable with parameters $m+\delta+1$ and $1$. Note that, based on the construction of the \polya\ point tree, conditionally on $\tilde{\emp} = (a, \Old)$, $d_{\tilde{\emp}}^{\rm(in)}$ is a mixed-Poisson random variable with mixing distribution $\Gamma_{\tilde{\emp}}(a^{\chi-1} - 1)$, where $\Gamma_{\tilde{\emp}}$ is a $\rm{Gamma}$ random variable with parameters $m + \delta + 1$ and $1$. Therefore,
	\begin{equation}
		\label{for:lem:size-bias:6}
		d_{\tilde{\emp}}^{\rm(in)} \mid \big\{ \tilde{\emp} = (a, \Old) \big\} \quad \overset{d}{=} \quad \Poi\big( \Gamma_{\rm{in}}(m+1)(a^{\chi-1} - 1) \big)~.
	\end{equation}
	Hence, the lemma follows immediately from \eqref{for:lem:size-bias:1}, \eqref{for:lem:size-bias:3}, \eqref{for:lem:size-bias:5}, and \eqref{for:lem:size-bias:6}.
\end{proof}
%\ch{Need to change here as well!}
%---------------------------------------------------------
\subsection{Convergence of thermodynamic quantities}\label{sec:thermodynamic-limit:subsec:convergence:thermodynamic-quantity}
%----------------------------------------------------------
We now proceed to prove Theorem~\ref{thm:limit:thermodynamic-quantities}. To prove this theorem, we utilise the fact that $\varphi(\beta, B)$ is a convex function with respect to both $\beta$ and $B$. The following lemma establishes this convexity property of $\varphi$:

\medskip
\begin{lemma}[Convexity of $\psi_n(\beta, B)$]\label{lem:convexity:pressure-limit}
	For any $\beta \in\R$ and $B \in \mathbb{R}$, let $\psi_n(\beta, B)$ be as defined in Theorem~\ref{thm:thermodynamic-limit:pressure}. Then, both $\beta \mapsto \psi_n(\beta, B)$ and $B \mapsto \psi_n(\beta, B)$ are convex functions of $\beta$ and $B$, respectively.
\end{lemma}
Proof to this lemma is trivial, since \(B\mapsto \psi_n(\beta,B)\) and \(\beta\mapsto \psi_n(\beta,B)\) can be interpreted as logs of moment generating functions.
To obtain the explicit expressions in Theorem~\ref{thm:limit:thermodynamic-quantities}, we use the following lemma. This lemma relies on the fact that preferential attachment models locally converge in probability to the P\'{o}lya point tree \cite[Theorem 1.1]{GHHR22}.
%We continue to prove results in this section subject to Lemma~\ref{lem:convexity:pressure-limit} and prove this lemma at the end of this section.

\medskip
\begin{lemma}\label{lem:limiting:edge-measure}
	Let $E_n$ denote the edge set of the preferential attachment model of size $n$ (denoted by $G_n$), and let $\mu_n$ denote the Boltzmann distribution on $G_n$. Then,
	\begin{align}
		\frac{1}{|E_n|}\sum_{\{i, j\} \in E_n}\langle \sigma_i \sigma_j \rangle_{\mu_n} &\overset{\prob}{\to} \E\left[ \frac{\tanh(\beta) + \tanh(\h(\bff(\hat{\emp}))) \tanh(\h(\bff(\hat{\emp }1)))}{1 + \tanh(\beta) \tanh(\h(\bff(\hat{\emp}))) \tanh(\h(\bff(\hat{\emp }1)))} \right]~, \label{eq:lem:limiting:edge-measure-1} \\
		\text{and} \qquad \frac{1}{n}\sum_{v \in [n]}\langle \sigma_v \rangle_{\mu_n} &\overset{\prob}{\to} \E[\tanh(\h(\bff(\emp)))]~. \label{eq:lem:limiting:edge-measure-2}
	\end{align}
\end{lemma}

The proof of Lemma~\ref{lem:limiting:edge-measure} follows similarly to the proof provided in \cite[Proof of Lemma~5.2]{DGvdH10}. Since this proof involves only minor modifications in the context of preferential attachment models, we defer the details to Appendix~\ref{appendix:results}. With Lemmas~\ref{lem:convexity:pressure-limit} and \ref{lem:limiting:edge-measure} in place, we now proceed to prove Theorem~\ref{thm:limit:thermodynamic-quantities}:

\begin{proof}[Proof of Theorem~\ref{thm:limit:thermodynamic-quantities}]
	First, we prove the theorem for the internal energy. %The proof for magnetisation follows in a similar manner. 
	To prove the convergence part, we use the convexity argument from Lemma~\ref{lem:convexity:pressure-limit}. Since $\beta \mapsto \psi_n(\beta, B)$ is a convex function in $\beta$, for any $\vep > 0$,
	\begin{equation}\label{for:prop:limiting:ie:01}
		\frac{1}{\vep}\left[ \psi_n(\beta, B) - \psi_n(\beta - \vep, B) \right] \leq \frac{\partial}{\partial \beta} \psi_n(\beta, B) \leq \frac{1}{\vep}\left[ \psi_n(\beta + \vep, B) - \psi_n(\beta, B) \right]~.
	\end{equation}
	Taking limits as $n \to \infty$ in \eqref{for:prop:limiting:ie:01}, and using Theorem~\ref{thm:thermodynamic-limit:pressure}, we obtain, for any $\vep > 0$,
	\begin{align}
		&\lim_{n \to \infty} \frac{1}{\vep} \left[ \psi_n(\beta, B) - \psi_n(\beta - \vep, B) \right] \overset{\prob}{\to} \frac{1}{\vep} \left[ \varphi(\beta, B) - \varphi(\beta - \vep, B) \right] \label{for:prop:limiting:ie:02}\\
		\text{and}\quad &\lim_{n \to \infty} \frac{1}{\vep} \left[ \psi_n(\beta + \vep, B) - \psi_n(\beta, B) \right] \overset{\prob}{\to} \frac{1}{\vep} \left[ \varphi(\beta + \vep, B) - \varphi(\beta, B) \right]~. \label{for:prop:limiting:ie:03}
	\end{align}
	Since the inequality in \eqref{for:prop:limiting:ie:01} holds for any $\vep > 0$, we take the limit $\vep \to 0$ to obtain $\frac{\partial}{\partial \beta} \varphi(\beta, B)$ as the limit of the RHS of both \eqref{for:prop:limiting:ie:02} and \eqref{for:prop:limiting:ie:03}. Therefore, by \eqref{for:prop:limiting:ie:01}, \eqref{for:prop:limiting:ie:02}, and \eqref{for:prop:limiting:ie:03}, 
	\begin{equation}\label{for:prop:limiting:ie:04}
		\frac{\partial}{\partial \beta} \psi_n(\beta, B) \overset{\prob}{\to} \frac{\partial}{\partial \beta} \varphi(\beta, B)~.
	\end{equation}
	We now need to prove the second part of \eqref{eq:def:internal:energy}. To do this, we follow the same strategy as in \cite[Lemma~5.2]{DGvdH10}. Let $E_n$ denote the edge set of the preferential attachment graph of size $n$. We simplify
	\begin{equation}\label{for:prop:limiting:ie:05}
		\frac{\partial}{\partial \beta} \psi_n(\beta, B) = \frac{1}{n} \sum_{(i, j) \in E_n} \langle \sigma_i \sigma_j \rangle_{\mu_n} = \frac{|E_n|}{n} \cdot \frac{1}{|E_n|} \sum_{(i, j) \in E_n} \langle \sigma_i \sigma_j \rangle_{\mu_n}~.
	\end{equation}
	Therefore, using \eqref{for:prop:limiting:ie:05} and Lemma~\ref{lem:limiting:edge-measure}, we obtain
	\begin{equation}\label{for:prop:limiting:ie:06}
		\frac{\partial}{\partial \beta} \psi_n(\beta, B) \overset{\prob}{\to} m \E \left[ \frac{\tanh(\beta) + \tanh(\h(\bff(\hat{\emp}))) \tanh(\h(\bff(\hat{\emp} 1)))}{1 + \tanh(\beta) \tanh(\h(\bff(\hat{\emp}))) \tanh(\h(\bff(\hat{\emp} 1)))} \right]~,
	\end{equation}
	proving Theorem~\ref{thm:internal:energy} (b).
	Proposition~\ref{prop:mag:lim} and Lemma~\ref{lem:h:magnetization}, completes the proof of Theorem~\ref{thm:magnetization} (a).
	%Similarly, from Lemma~\ref{lem:convexity:pressure-limit}, we obtain that $B \mapsto \psi_n(\beta, B)$ is a convex function. Next, using a similar calculation as done in \eqref{for:prop:limiting:ie:01}-\eqref{for:prop:limiting:ie:03}, 
	%\begin{equation}\label{for:prop:limiting:ie:07}
	%	\frac{\partial}{\partial B} \psi_n(\beta, B) \overset{\prob}{\to} \frac{\partial}{\partial B} \varphi(\beta, B)~.
	%\end{equation}
	%Therefore, by Lemma~\ref{lem:limiting:edge-measure}, we obtain the RHS of \eqref{eq:lem:limiting:edge-measure-2} as the thermodynamic limit of magnetisation per vertex.
\end{proof}
\begin{Remark}[Similarity with configuration model result]
	\rm Note that, from the description of the P\'{o}lya point tree, the expected degree of the root in the P\'{o}lya point tree is $2m$. Therefore, substituting $m$ with $\E[D(\emp)]/2$ in \eqref{eq:def:internal:energy}, we obtain a result similar in flavour to that obtained in \cite[Corollary~1.6(b)]{DGvdH10} for the configuration model. Similarly, using the distributional recursion property of the $\h$ random variables in \eqref{eq:distributional-recursion:1}, it can be shown that the explicit expression in \eqref{eq:def:magnetization} is the same as the one obtained in \cite[Corollary~1.6(a)]{DGvdH10}.\hfill$\blacksquare$
\end{Remark}
%\medskip

%%%%%%%%%%%%%%%%%%%%%%%%%%%%%%%%%%%%%%%%%%%%%%
\section{Inverse critical temperature}\label{sec:inv:critical-temperature}
%%%%%%%%%%%%%%%%%%%%%%%%%%%%%%%%%%%%%%%%%%%%%%
In this section, we prove Theorem~\ref{thm:inv:critical-temp:PA}, which identifies the critical temperature for the phase transition in the Ising model on preferential attachment models. The proof of Theorem~\ref{thm:inv:critical-temp:PA} is based on the inverse critical temperature for the P\'{o}lya point tree.

We first define the inverse critical temperature for a rooted tree.
Let $\mu^{\sss (\beta,B)}$ denote the Boltzmann distribution on a tree $(\Tree, o)$ with inverse temperature parameter $\beta$ and external magnetic field $B$. Define
\begin{equation}
	\label{eq:def:Boltzmann:zero-mag-field}
	\mu^{\sss (\beta,0^+)}(\cdot) = \lim_{B \searrow 0} \mu^{\sss (\beta,B)}(\cdot)~.
\end{equation}
The \emph{inverse critical temperature} for a rooted tree $(\Tree, o)$, denoted $\beta_c(\Tree, o)$, is defined as
\begin{equation}
	\label{eq:def:inverse-critical-temperature}
	\beta_c(\Tree, o) = \inf\left\{ \beta \mid \mu^{\sss (\beta,0^+)}(\sigma_o = +1) - \mu^{\sss (\beta,0^+)}(\sigma_o = -1) > 0 \right\}~.
\end{equation}
%where $\mu^{\sss (\beta,0^+)}(\cdot)$ is the Boltzmann distribution on $(\Tree, o)$ under zero external magnetic field, as defined in \eqref{eq:def:Boltzmann:zero-mag-field}.

The following proposition identifies the inverse critical temperature for the P\'{o}lya point tree with parameters $m \geq 2$ and $\delta > 0$:

\medskip
\begin{Proposition}[Inverse critical temperature for the PPT]
	\label{prop:inv:critical-temp:PPT}
	Fix $m \in \mathbb{N} \setminus \{1\}$ and $\delta > -m$. Then, the inverse critical temperature for the P\'{o}lya point tree with parameters $m$ and $\delta$ is given by $\beta_c(m, \delta)$ as defined in Theorem~\ref{thm:inv:critical-temp:PA}.
\end{Proposition}
\smallskip
First, we prove Theorem~\ref{thm:inv:critical-temp:PA} using Proposition~\ref{prop:inv:critical-temp:PPT}, and then we proceed to prove Proposition~\ref{prop:inv:critical-temp:PPT} in detail.

\begin{proof}[Proof of Theorem~\ref{thm:inv:critical-temp:PA} subject to Proposition~\ref{prop:inv:critical-temp:PPT}]
	Proposition~\ref{prop:inv:critical-temp:PPT} identifies $\beta_c(m,\delta)$ as the inverse critical temperature for the P\'{o}lya point tree with parameters $m$ and $\delta>-m$. Therefore,
	\begin{align}
		\mu^{\sss (\beta,0^+)}(\sigma_o=+1) - \mu^{\sss (\beta,0^+)}(\sigma_o=-1) &> 0, \quad \text{for all } \beta > \beta_c(m,\delta), \label{for:thm:PA:inv-temp:1-1} \\
		\text{and} \quad \mu^{\sss (\beta,0^+)}(\sigma_o=+1) - \mu^{\sss (\beta,0^+)}(\sigma_o=-1) &= 0, \quad \text{for all } \beta < \beta_c(m,\delta). \label{for:thm:PA:inv-temp:1-2}
	\end{align}
	Note that, by Theorem~\ref{thm:magnetization} and Lemma~\ref{lem:h:magnetization},
	\begin{equation}\label{for:thm:PA:inv-temp:2}
		\mu^{\sss (\beta,0^+)}(\sigma_o=+1) - \mu^{\sss (\beta,0^+)}(\sigma_o=-1) = \lim_{B \searrow 0} \mathbb{E}[\magn{\emp \to}] = \lim_{B \searrow 0} M(\beta,B).
	\end{equation}
	Therefore, from \eqref{for:thm:PA:inv-temp:1-1}, \eqref{for:thm:PA:inv-temp:1-2}, and \eqref{for:thm:PA:inv-temp:2}, we obtain
	\begin{equation}\label{for:thm:PA:inv-temp:3}
		\beta_c(m,\delta) = \inf \Big\{ {\beta}: \lim_{B \searrow 0} M(\beta,B) > 0 \Big\},
	\end{equation}
	proving that $\beta_c(m,\delta)$ is the inverse critical temperature for preferential attachment models with parameters $m$ and $\delta$.
\end{proof}

The proof of Proposition~\ref{prop:inv:critical-temp:PPT} follows from \cite[Theorem~2.1]{Lyo89}, with adaptations to our specific setting. The proof employs concepts such as the branching number and the growth of trees. We recall the definitions of the branching number and the growth of trees as established in \cite{Lyo89,Lyo90}, using cut-set. In context of trees, a \emph{cut-set} is a set of vertices of the tree such that it  has non-zero intersection with any infinite path starting at the root.   

\medskip
\begin{definition}[Branching number and growth of tree]\label{def:branching-number:growth}
	The \emph{branching number} of a tree $\Tree$, denoted by $\rm{br}(\Tree)$, is defined as
	\eqn{
	\rm{br}(\Tree)=~\inf\Big\{ \lambda>0:\liminf\limits_{\Pi\to\infty}\sum\limits_{u\in\Pi} \lambda^{-|u|}=0 \Big\}~,
	}
	where \(\Pi\) is a cut-set of the tree \(\Tree\), and \(\Pi\to \infty\) implies the fact that \( \inf\{{\rm d}(o,v):v\in\Pi\}\to \infty\). It can also be expressed as 
	\eqn{\label{eq:def:branching-number}
	\rm{br}(\Tree)=~\sup\Big\{ \lambda>0:\liminf\limits_{\Pi\to\infty}\sum\limits_{u\in\Pi} \lambda^{-|u|}=\infty \Big\}
	=~\inf\Big\{ \lambda>0:\inf\limits_{\Pi}\sum\limits_{u\in\Pi} \lambda^{-|u|} \Big\}
	~,}
	whereas the \emph{growth} of $\Tree$, denoted by $\rm{gr}(\Tree)$, is defined as
	\eqn{\label{eq:def:growth}
	{\rm{gr}}(\Tree)=\inf\Big\{ \lambda>0:\liminf\limits_{n\to \infty} {M_n}\lambda^{-n}=0\Big\}=\liminf\limits_{n\to\infty}M_n^{1/n}
	~,}
	where $M_n$ is the size of the $n$-th generation in the tree $\Tree$. \hfill\(\blacksquare\)
\end{definition}
\smallskip
The notation used in \cite{Lyo89} differs slightly from that used here. Adapting to our notation, we can deduce from \cite[Theorem~2.1]{Lyo89} that
\begin{equation}\label{for:thm:inv:critical-temp:1}
	\tanh\big(\beta_c(m,\delta)\big) = \frac{1}{{\rm{br}}\big(\PPT(m,\delta)\big)},
\end{equation}
where ${\rm{br}}\big(\PPT(m,\delta)\big)$ is the branching number of the P\'{o}lya point tree with parameters $m$ and $\delta$. We now identify the explicit expression for ${\rm{br}}\big(\PPT(m,\delta)\big)$. We prove and use the following proposition, which is similar in spirit to \cite[Proposition~6.4]{Lyo89}:
\medskip
\begin{Proposition}[Bound for branching number of the P\'{o}lya point tree]\label{prop:branching-number:PPT}
	%Let $\bfT_{\kappa}$ denote the mean offspring operator of the P\'{o}lya point tree with parameters $m$ and $\delta$. 
	For $m \geq 2$ and $\delta > -m$,
	\begin{equation}\label{eq:prop:branching-number:PPT}
		\frac{1}{{\rm{br}}\big(\PPT(m,\delta)\big)} \leq \pi_c (m,\delta)~\mbox{a.s.},
	\end{equation}
	where $\pi_c (m,\delta)$ is the critical percolation threshold for \(\PPT(m,\delta)\), defined in Section~\ref{sec:RPPT}.
\end{Proposition}
\medskip
\begin{proof}[Proof of Proposition~\ref{prop:branching-number:PPT}]
To prove the proposition, we follow a strategy similar to that in \cite[Theorem~6.2]{Lyo89}. \cite[Theorem~1.2]{HHR23} shows that for $m \geq 2$ and $\delta > 0$, the critical percolation threshold for $\PPT(m,\delta)$ is ${1}/{r(\bfT_{\kappa})}$. We now prove that $\PPT(m,\delta)$ almost surely dies out when percolated with $\pi < {1}/{{\rm{br}}\big(\PPT(m,\delta)\big)}$. Therefore, the critical percolation threshold of $\PPT(m,\delta)$ is at least ${1}/{{\rm{br}}\big(\PPT(m,\delta)\big)}$, completing the proof of Proposition~\ref{prop:branching-number:PPT}.

To reduce notational complexity, we shall refer $\PPT(m,\delta)$ by $\Tree$ in this proof, and for any $\pi\in[0,1]$, we use $\Tree(\pi)$ to denote the $\pi$ percolated $\PPT(m,\delta)$.
	For any cutset $\Pi$ and $\lambda>0$, let
	\eqn{\label{for:prop:branching-number:1}
		Z_{\Pi}(\lambda)=\sum\limits_{u\in\Pi}\lambda^{-|u|}~,}
	and for any $\pi\in[0,1],$ define
	\eqn{\label{for:prop:branching-number:2}
		Z_{\Pi(\pi)}(\lambda)=\sum\limits_{u\in\Pi(\pi)}\lambda^{-|u|}~,}
	where $\Pi(\pi)=\Pi\cap\Tree(\pi).$ Then,
	\eqn{\label{for:prop:branching-number:3}
		\E_{\pi}\Big[ Z_{\Pi(\pi)}(1) \Big]=\sum\limits_{u\in\Pi} \prob_{\Tree}\big( u\in\Tree(\pi) \big)=\sum\limits_{u\in \Pi}\pi^{|u|}=Z_{\Pi}(1/\pi)~,}
	where $\E_{\pi}$ is the expectation with respect to the percolation.
	If $1/\pi>\br{\Tree}$, then there exists a sequence $\Pi_n\to\infty$ such that $Z_{\Pi_n}(1/\pi)\to0.$ Hence, from Fatou's lemma applied to \eqref{for:prop:branching-number:3}, we obtain conditionally on $\Tree$,
	\eqn{\label{for:prop:branching-number:4}
		\liminf\limits_{n\to\infty} Z_{\Pi_n(\pi)}(1)=0\quad \text{a.s.}}
	Therefore, conditionally on $\Tree,~\Tree(\pi)$ is finite almost surely. From definition of the critical percolation threshold $\pi_c(m,\delta)$ in \cite{HHR23}, we can upper bound the random variable $1/\br{\Tree}$ by $\pi_c$. 
\end{proof}

\begin{remark}[Lower bound of \(\br{\PPT(m,\delta)}\)]
\rm \cite[Theorem~1.2]{HHR23} further identifies that for \(\delta>0\), $\pi_c(m,\delta)$ of P\'{o}lya point tree is $1/r(\bfT_{\kappa})$, whereas for \(\delta\leq 0,~\pi_c(m,\delta)\) is \(0\). Hence, for \(\delta>0\)
\eqn{\label{for:prop:branching-number:5}
		\frac{1}{\br{\Tree}}\leq \pi_c (m,\delta)=\frac{1}{r(\bfT_{\kappa})}\quad\text{a.s.}~,}
 and for \(\delta\leq 0,~1/{\br{\PPT(m,\delta)}}=0\). Therefore, for \(\delta>0\),  we obtain \(r(\bfT_{\kappa})\) as the almost sure lower bound of $\br{\Tree}$. \hfill\(\blacksquare\)
\end{remark}
\smallskip
Lemma~\ref{lem:growth:PPT} proves that $\gr{\Tree}$ is at most $r(\bfT_{\kappa})$ almost surely, which serves as an almost sure upper bound for $\br{\Tree}$ as well. Hence, $\br{\PPT(m,\delta)}$ is $r(\bfT_{\kappa})$ almost surely.
On the other hand, based on the definitions of growth and branching numbers for a tree, it follows that ${\rm{br}}\big(\PPT(m,\delta)\big)\leq{\rm{gr}}\big(\PPT(m,\delta)\big)$. The next lemma computes an upper bound for ${\rm{gr}}\big(\PPT(m,\delta)\big)$:
\medskip
\begin{lemma}[Growth of the P\'{o}lya point tree]\label{lem:growth:PPT}
	For $m \geq 2$ and $\delta > 0$,
	\begin{equation}\label{eq:lem:growth:PPT}
		{\rm{gr}}\big(\PPT(m,\delta)\big) \leq r(\bfT_{\kappa}) \text{ a.s.}
	\end{equation}
\end{lemma}

%According to Lemma~\ref{lem:growth:PPT}, ${\rm{br}}\big(\PPT(m,\delta)\big)$ is at most $r(\bfT_{\kappa})$.  With these results, we are now equipped to prove Proposition~\ref{prop:branching-number:PPT}.

%\medskip
%Now, we are left to prove Lemma~\ref{lem:growth:PPT}. 
\begin{proof}[Proof of Lemma~\ref{lem:growth:PPT}]
	For proving this lemma, we first show that $\E\big[ \gr{\PPT(m,\delta)} \big]\geq r(\bfT_{\kappa})$, and then we prove $\gr{\PPT(m,\delta)}$ is at least $r(\bfT_{\kappa})$ almost surely. 
Let $M_n(x,s)$ denote the number of nodes in $\PPT(m,\delta)$ rooted at $(x,s)\in\Scal$. From the definition in \eqref{eq:def:growth},
	\begin{equation}\label{for:lem:growth:PPT:0}
		\E\big[ \gr{\PPT(m,\delta)} \big] = \E\big[\liminf\limits_{n\to \infty} M_n(\bff(\emp))^{1/n}\big] \leq \liminf\limits_{n\to \infty} \E\Big[ M_n(\bff(\emp))^{1/n} \Big]~,
	\end{equation}
	where the inequality follows from Fatou's lemma. Using the fact that $x\mapsto x^{1/n}$ is a concave function, we apply Jensen's inequality to obtain
	\begin{equation}\label{for:lem:growth:PPT:1}
		\E\big[ \gr{\PPT(m,\delta)} \big] \leq \liminf\limits_{n\to \infty} \E\Big[ M_n(\bff(\emp)) \Big]^{1/n}~.
	\end{equation}
	
	We first aim to show that the LHS of \eqref{for:lem:growth:PPT:1} is upper bounded by $r(\bfT_{\kappa})$. To prove this, we show that
	\begin{equation}\label{for:lem:growth:PPT:2}
		M_n(\bff(\emp)) \overset{d}{=} M_n(U_1, \Young) + M_{n-1}(U_1 U_2^{1/\chi}, \Old)~,
	\end{equation}
	where $U_1, U_2$ are i.i.d.\ $\Unif[0,1]$ random variables. To establish this result, we view the P\'{o}lya point tree rooted at $\emp$ from the perspective of a uniformly chosen out-edge of the root $\emp$. 
	
	Let $\emp u$ be the uniformly chosen $\Old$ neighbour of $\emp$. Let $M^{(-u)}_n(\bff(\emp))$ denote the number of nodes in the $n$-th generation of $\PPT$ rooted at $\emp$, ignoring the size of the sub-tree rooted at $\emp u$. Note that, by the construction of the P\'{o}lya point tree, conditionally on the age of $\emp$, $M_{n-1}(\bff(\emp u))$ and $M^{(-u)}_n(\bff(\emp))$ are independent. Conditionally on the age of $\emp$, $A_\emp=a$, the sub-tree rooted at $\emp$, excluding the sub-tree rooted at $\emp u$, has $m-1$ many $\Old$-labelled children in the first generation. Furthermore, since the strength of any $\Young$-labelled node is identically distributed to the root, it follows that, conditionally on $A_\emp=a$, $\emp$ and $(a, \Young)$ have identical offspring distributions. Hence, conditionally on $A_\emp=a$,
	\begin{equation}\label{for:lem:growth:PPT:3}
		M_n(a, \Young) \overset{d}{=} M^{(-u)}_n(\bff(\emp))~.
	\end{equation}
	
	On the other hand, conditionally on $A_\emp=a$, the age of $\emp U$ is distributed as $a U_2^{1/\chi}$, where $U_2 \sim \Unif[0,1]$ and labelled $\Old$, and is independent of $A_\emp$. Therefore, conditionally on $A_\emp=a$,
	\begin{equation}\label{for:lem:growth:PPT:4}
		M_{n-1}(\bff(\emp u)) \overset{d}{=} M_{n-1}(a U_2^{1/\chi}, \Old)~.
	\end{equation}
	
	Since $A_\emp \sim \Unif[0,1]$ and $M^{(-u)}_n(\bff(\emp))$ and $M_{n-1}(\bff(\emp u))$ are mutually independent and independent of $A_\emp$, we obtain
	\begin{equation}\label{for:lem:growth:PPT:5}
		\left( M_{n-1}(\bff(\emp u)), M^{(-u)}_n(\bff(\emp)) \right) \overset{d}{=} \left( M_{n-1}\big(U_1 U_2^{1/\chi}, \Old\big), M_{n}\big(U_1, \Young\big) \right)~,
	\end{equation}
	where $U_1 \sim \Unif[0,1]$ and is independent of $U_2$. Thus, \eqref{for:lem:growth:PPT:1} follows immediately.
	
Recall that the state-space of the P\'olya point tree is \(\Scal=[0,1]\times\{\Old,\Young\}\cup [0,1]\). The last \([0,1]\) part in the definition of \(\Scal\) arises due to the root having no label. Note that, every node in the P\'olya point tree has type in \( [0,1]\times\{\Old,\Young\} \). Since in the assymptotic computations the effect of root becomes negligible, we shall consider the sub-trees of the P\'olya point tree rooted at the offspring of \(\emp\) from now on. Next, for any $(x,s) \in \Scal_c=[0,1]\times\{\Old,\Young\}$ and $n \in \N$,
	\begin{equation}\label{for:lem:growth:PPT:6}
		\E[M_n(x,s)] = \langle \one_{(x,s)}, \bfT_{\kappa}^n \ubar{\one} \rangle_{\Scal},
	\end{equation}
	where \( \one_{(x,s)} \) puts unit mass at \( (x,s)\in\Scal_c\) and \(0\) elsewhere, and $\ubar{\one}(y,t) \equiv 1$ for all $(y,t) \in \Scal_c$, and $\langle f, g \rangle_{\Scal_c} = \int_{\Scal_c} f(z) g(z) \, dz$ for any functions $f$ and $g$ from $\Scal_c$ to $\R$. Let $\kappa$ denote the mean offspring generator for the P\'{o}lya point tree with parameters $m$ and $\delta$. Recall that $\Scal_c=[0,1]\times\{\Old,\Young\}$. In \cite{HHR23}, together with Hazra, we proved that the integral operator $\bfT_{\kappa}$, with kernel $\kappa$, admits an eigenfunction $h$ corresponding to its spectral norm $r(\bfT_{\kappa})$ in the extended type-space $\Scal_e=[0,\infty)\times\{\Old,\Young\}$. The eigenfunction $h$ is given by
	\[
		h(x,s)=\frac{\bfp_s}{\sqrt{x}},~\mbox{for all }(x,s)\in\Scal_e~, 
	\]
	where $\bfp=(\bfp_{\old},\bfp_{\young})$ is as defined in \cite[Theorem~2.5]{HHR23}.
	
	Since $\kappa(u,v) \geq 0$ for all $u, v \in \Scal_c$, it can further be shown that
	$\bfT_{\kappa} f(u) \geq \bfT_{\kappa} g(u)$ for all $u \in \Scal_c$, if $f(v) \geq g(v)$ for all $v \in \Scal_c$. Note that, $h(x,s) \geq 1$, for all $(x,s) \in \Scal_c$, where $\bfc = \max\{1/\bfp_\old, 1/\bfp_\young\}$. Therefore,
	\begin{equation}\label{for:lem:growth:PPT:7}
		\langle \one_{(x,s)}, \bfT_{\kappa}^n \ubar{\one} \rangle_{\Scal_c} \leq \bfc \langle \one_{(x,s)}, \bfT_{\kappa}^n h \rangle_{\Scal_c} \leq \bfc \langle \one_{(x,s)}, \bar{\bfT}_{\kappa}^n h \rangle_{\Scal_e} = \bfc r(\bfT_{\kappa})^n h(x,s)~.
	\end{equation}
	
	Since $r(\bfT_{\kappa}) > 1$, from \eqref{for:lem:growth:PPT:2}, \eqref{for:lem:growth:PPT:6}, and \eqref{for:lem:growth:PPT:7}, we obtain
	\begin{equation}\label{for:lem:growth:PPT:8}
		\E[M_n(\bff(\emp))] \leq \bfc r(\bfT_{\kappa})^n \left[ \E[h(U_1, \Young)] + \E\left[h\big(U_1 U_2^{1/\chi}, \Young\big)\right] \right] \leq \bfc_1 r(\bfT_{\kappa})^n~,
	\end{equation}
	where $\bfc_1 = \bfc/2$. Plugging this upper bound for $\E[M_n(\bff(\emp))]$ into \eqref{for:lem:growth:PPT:1}, we obtain $r(\bfT_{\kappa})$ as an almost sure upper bound for $\gr{\PPT(m,\delta)}$. 
	%Remark~\ref{remark:lower:bound:branching-number} explains that $\br{\PPT(m,\delta)}$ has $r(\bfT_{\kappa})$ as an almost sure lower bound. Since the branching number lower bounds the growth of a tree almost surely, Lemma~\ref{lem:growth:PPT} proves that $r(\bfT_{\kappa})$ is also an almost sure lower bound for $\gr{\PPT(m,\delta)}$. Hence, $\gr{\PPT(m,\delta)}$ is $r(\bfT_{\kappa})$ almost surely.
\end{proof}

Now that we have all the necessary results, we can prove Proposition~\ref{prop:inv:critical-temp:PPT}. We use \cite[{Theorem~2.1}]{Lyo89} to establish the result. Note that the notations used in \cite{Lyo89} differ from those in this paper. Therefore, to apply the conclusions from \cite{Lyo89}, we equate \( \frac{J}{kT} \) in \cite{Lyo89} with \( \beta \) in this paper.

\begin{proof}[Proof of Proposition~\ref{prop:inv:critical-temp:PPT}]
	By Proposition~\ref{prop:branching-number:PPT}, it follows immediately that for \(\delta\leq 0,~\beta_c(m,\delta)=0\).	
	From the definition of branching number and growth of a tree, a.s.\,,
	\begin{equation}\label{for:thm:inv:temp:00}
		\br{\Tree}\leq\gr{\Tree}~ \mbox{for all infinite tree }\Tree~.
	\end{equation}
	By Lemma~\ref{lem:growth:PPT}, Proposition~\ref{prop:branching-number:PPT} and \eqref{for:thm:inv:temp:00}, for \(\delta>0,\)
	\eqn{\label{for:thm:inv:temp:01}
	\br{\PPT(m,\delta)}=r(\bfT_{\kappa})
	%=\frac{2\big( m(m+\delta)+\sqrt{m(m-1)(m+\delta)(m+\delta+1)} \big)}{\delta}
	~.}
	By \cite[Theorem~2.1]{Lyo89}, for \(\delta>0\),
%	\begin{equation}\label{for:thm:inv:temp:1}
%		\tanh\left( \frac{J}{kT_c} \right)\br{\PPT(m,\delta)} = 1~,
%	\end{equation}
%	where \( T_c \) is the critical temperature described in \cite{Lyo89}. Equating these notations to our settings, we obtain
	\begin{equation}\label{for:thm:inv:temp:2}
		\tanh(\beta_c(m,\delta))\br{\PPT(m,\delta)} = 1~.
	\end{equation}
	By \eqref{for:thm:inv:temp:01}, for \(\delta>0\),
	\begin{equation}\label{for:thm:inv:temp:3}
		\beta_c (m,\delta) = \arctanh\left( \frac{1}{r(\bfT_{\kappa})} \right)~.
	\end{equation}
	Substituting the explicit value of \( r(\bfT_{\kappa}) \) from \cite[Theorem~1.2]{HHR23} into \eqref{for:thm:inv:temp:3}, we obtain Proposition~\ref{prop:inv:critical-temp:PPT} immediately.
\end{proof}

%%%%%%%%%%%%%%%%%%%%%%%%%%%%%%%%%%%%%%%%%%%%%%%%%%%

%%%%%%%%%%%%%%%%%%%%%%%%%%%%%%%%%%%%%%%%%%%%%%%%%%

%\section{Conclusion}\label{sec13}

%Conclusions may be used to restate your hypothesis or research question, restate your major findings, explain the relevance and the added value of your work, highlight any limitations of your study, describe future directions for research and recommendations. 

%In some disciplines use of Discussion or 'Conclusion' is interchangeable. It is not mandatory to use both. Please refer to Journal-level guidance for any specific requirements. 

\backmatter
%\newpage
%\journalversion{
%\bmhead{Supplementary information}
%We have added the arXiv version of the article as a supplementary document. In several places of this article, we have referred to the appendix of this supplementary document.}
%If your article has accompanying supplementary file/s please state so here. }

%Authors reporting data from electrophoretic gels and blots should supply the full unprocessed scans for key as part of their Supplementary information. This may be requested by the editorial team/s if it is missing.
%Please refer to Journal-level guidance for any specific requirements.
%\parbox{0.8\textwidth}

%\indent
\bmhead{Acknowledgments}
\begin{wrapfigure}{r}{0.2\textwidth}
	\centering
	\includegraphics[width=\linewidth]{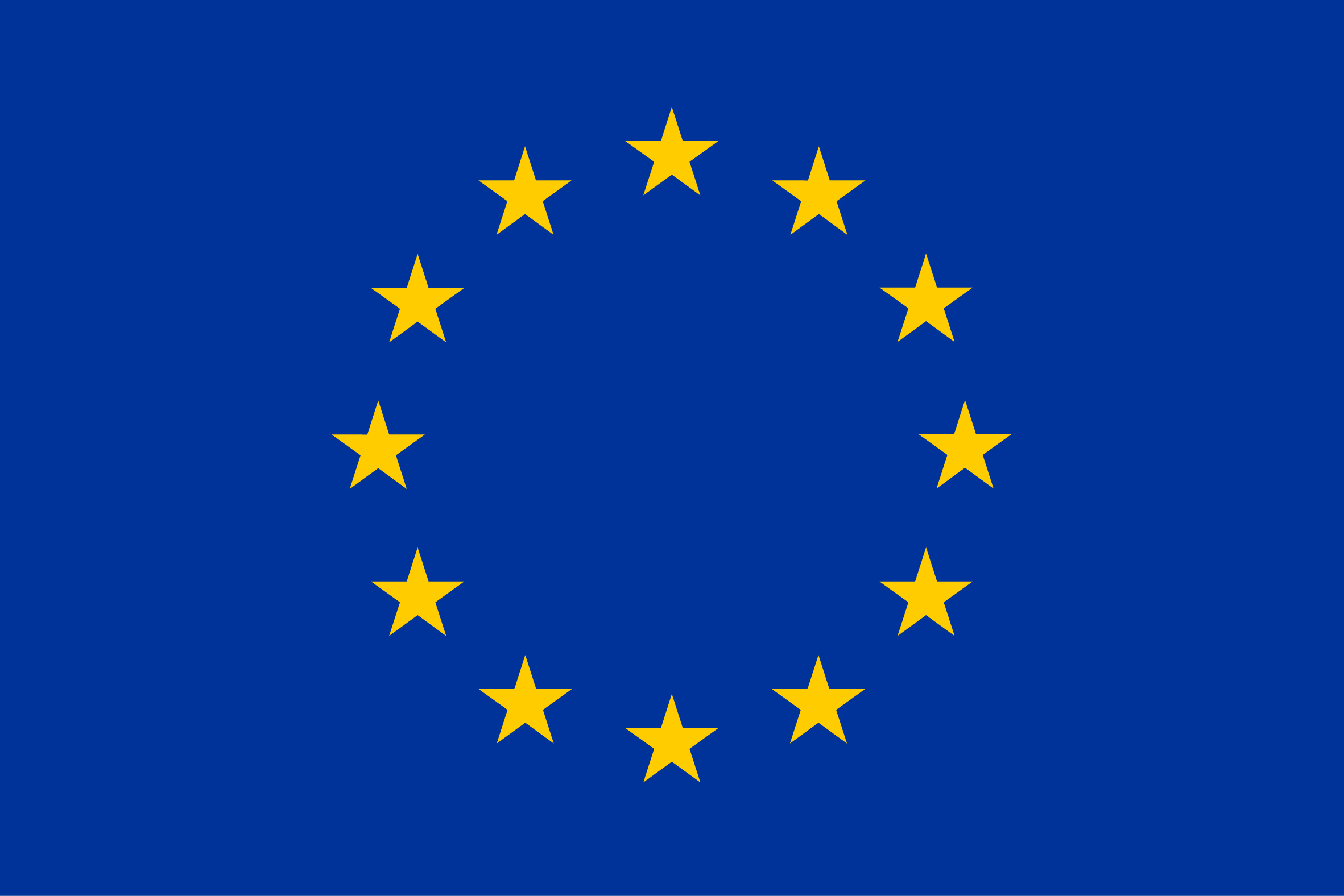} 
	\end{wrapfigure}
This work is supported in part by the Netherlands Organisation for Scientific Research (NWO) through the Gravitation {\sc Networks} grant 024.002.003 and the National Science Foundation under Grant No. DMS-1928930, while the authors were in residence at the Simons Laufer Mathematical Sciences Institute in Berkeley, California, during the Spring, 2025 semester. The work of RR is further supported by the European Union's Horizon 2020 research and innovation programme under the Marie Sk\l{}odowska-Curie grant agreement no.\ 945045.
{}
%\parbox{0.12\textwidth}
%{~~~~\includegraphics[width=0.12\textwidth]{EUlogo.jpg}}

%\section*{Declarations}

%Some journals require declarations to be submitted in a standardised format. Please check the Instructions for Authors of the journal to which you are submitting to see if you need to complete this section. If yes, your manuscript must contain the following sections under the heading `Declarations':

%\begin{itemize}
%\item Conflict of interest/Competing interests (check journal-specific guidelines for which heading to use)
%\item Ethics approval 
%\item Consent for publication
%\item Availability of data and materials
%\item Authors' contributions
%\end{itemize}

%\noindent
%If any of the sections are not relevant to your manuscript, please include the heading and write `Not applicable' for that section. 


%\arxiversion{
\begin{appendices}
%\section{Section title of first appendix}\label{secA1}
% !TEX root = ising.tex
%\appendix
%%%%%%%%%%%%%%%%%%%%%%%%%%%%%%%%%%%%%%%%%%%%%%%%%%
\section{Remaining proofs}\label{appendix:results}
%%%%%%%%%%%%%%%%%%%%%%%%%%%%%%%%%%%%%%%%%%%%%%%%%%
First, we prove Lemma~\ref{lem:h:magnetization}. This proof follows a similar approach to that in \cite[Proposition~5.13]{RvdHSF}, with adaptations made for a more general setting. Our goal is to demonstrate that \eqref{eq:prop:root-magnetization:1} is the unique fixed-point solution to the recursive equations in \eqref{eq:test:1}.

\begin{proof}[Proof of Lemma~\ref{lem:h:magnetization}]
	We prove Lemma~\ref{lem:h:magnetization} for a very general case where the tree is locally finite almost surely. Let $\Tree_{\ell}(\omega)$ denote $\Tree$ rooted at \(o\), such that $\bff(o)=\omega \in \Scal$, truncated at depth $\ell$, and let $\magn{\omega,\ell,+/f}$ denote the root magnetization given $\Tree_{\ell}(\omega)$ with an external magnetic field per vertex $B > 0$, under either the $+$ or free boundary condition. Since $\Tree$ is locally finite almost surely, by \cite[Lemma~3.1]{DGvdH10} there exists $M < \infty$ such that for all $\ell \geq 1$,
	\begin{equation}
		\label{for:lem-h:1}
		|\magn{\omega,\ell,+} - \magn{\omega,\ell,f}| \leq \frac{M}{\ell}, \quad \text{a.s. for all } \omega \in \Scal.
	\end{equation}
	Therefore, as $\ell \to \infty$, both $\magn{\omega,\ell,+}$ and $\magn{\omega,\ell,f}$ converge to the same limit, which we denote by $\magn{o \to}$, with \(\bff(o)=\omega\), defined in Section~\ref{sec:thermodynamic-limit:subsec:belief-propagation}.
	
	Note that, for all $\omega \in \Scal$, $\underline{\h}^{(\ell)}(\omega) \equiv \arctanh\big( \magn{\omega,\ell,f} \big)$, with initialisation $\underline{\h}^{(0)}(\omega) \equiv B$ for all $\omega \in \Scal$, satisfies the recursive equation in \eqref{eq:test:1} due to Lemma~\ref{lem:tree-pruning}. By the GKS inequality in Lemma~\ref{lem:GKS}, $\magn{\omega,\ell,f}$ is monotonically increasing in $\ell$. Furthermore, from \eqref{eq:test:1}, for all $\omega \in \Scal$, and $\ell\geq 1$,
	\[
	B = \underline{\h}^{(0)}(\omega) \leq \underline{\h}^{(\ell)}(\omega) \leq B + D(\emp) < \infty,
	\]
	for all $\ell \geq 1$, almost surely. Therefore, by the monotone convergence theorem, $\underline{\h}^{(\ell)}(\omega)$ converges to some limit $\underline{\h}(\omega)$ almost surely. Hence, $\underline{\h}$ is a fixed point of \eqref{eq:test:1} by \cite[{Proof of Lemma~2.3}]{DM10}.
	
	Similarly, for all $\omega \in \Scal,~\bar{\h}^{(\ell)}(\omega) \equiv \magn{\omega,\ell,+}$ also satisfies \eqref{eq:test:1}, with initialisation $\bar{\h}^{(0)}(\omega) \equiv \infty$ for all $\omega \in \Scal$. Then, $\bar{\h}^{(\ell)}(\omega)$ is monotonically decreasing and, for all $\omega \in \Scal$, and $\ell\geq1$,
	\[
	B \leq \bar{\h}^{(\ell)}(\omega) \leq \bar{\h}^{(1)}(\omega) = B + D(\emp) < \infty,
	\]
	almost surely. Therefore, $\bar{\h}^{(\ell)}(\omega)$ also converges to some limit $\bar{\h}(\omega)$ for all $\omega \in \Scal$.
	Therefore, by \cite[Lemma~3.1]{DGvdH10}, for all $\omega \in \Scal$,
	\begin{equation}
		\label{for:lem:h-uniqueness:8}
		\big| \tanh\big(\underline{\h}^{(\ell)}(\omega)\big) - \tanh\big( \bar{\h}^{(\ell)}(\omega) \big) \big| \leq \big| \magn{\omega,\ell,f} - \magn{\omega,\ell,+} \big| \to 0, \quad \text{as}~ \ell \to \infty,
	\end{equation}
	and this holds almost surely, proving that both the limits are unique, and $\h(\omega) = \arctanh\big(\magn{\omega \to}\big)$ is the unique fixed point solution of \eqref{eq:test:1}.
	%Let $\h$ be a fixed point of \eqref{eq:test:1}, and $\h_\star^{(0)} \equiv \h$. Then, $\h^{(\ell)}_\star$ also converges to some limit $\h_\star$ when applying \eqref{eq:test:1}. Note that for all $\omega \in \Scal,~\underline{\h}^{(0)}(\omega) \leq \h_\star^{(0)}(\omega) \leq \bar{\h}^{(0)}(\omega)$. Now, coupling $\underline{\h}^{(\ell)}(\omega)$, $\h_\star^{(\ell)}(\omega)$, and $\bar{\h}^{(\ell)}(\omega)$ with the same sequence of $D_{\ell}(\omega)$ while applying \eqref{eq:test:1}, the inequality is preserved by the GKS inequality as
	%\[
	%\underline{\h}^{(\ell)}(\omega) \leq \h_\star^{(\ell)}(\omega) \leq \bar{\h}^{(\ell)}(\omega),
	%\]
	%for all $\omega \in \Scal$ and $\ell \geq 1$. 
	 %Since \eqref{for:lem:h-uniqueness:8} holds almost surely and for any realisation of $\h_\star$, for all $\omega \in \Scal,~\underline{\h}(\omega),~\h_\star(\omega)$, and $\bar{\h}(\omega)$ are equal in distribution, proving that $\h(\omega) = \arctanh\big(\magn{\omega \to}\big)$ is the unique fixed point solution of \eqref{eq:test:1}.
\end{proof}

Lemma~\ref{lem:limiting:edge-measure} is a direct consequence of the local convergence of preferential attachment models to the P\'{o}lya point tree:

\begin{proof}[Proof of Lemma~\ref{lem:limiting:edge-measure}]
	The LHS of \eqref{eq:lem:limiting:edge-measure-1} can be viewed as the expectation of the correlation $\langle \sigma_u \sigma_v \rangle_{\mu_n}$ with respect to a uniformly chosen edge $\{u,v\}$. For a uniformly chosen edge $\{u,v\}$, denote by $B_{\{u,v\}}(\ell)$ the subgraph of $G_n$ consisting of all vertices at a distance of at most $\ell$ from either $u$ or $v$. Therefore, by the GKS inequality in Lemma~\ref{lem:GKS},
	\begin{equation}
		\label{for:lem:limiting:edge-measure:1}
		\langle \sigma_u \sigma_v \rangle_{B_{\{u,v\}}(\ell)}^f \leq \langle \sigma_u \sigma_v \rangle_{\mu_n} \leq \langle \sigma_u \sigma_v \rangle_{B_{\{u,v\}}(\ell)}^+,
	\end{equation}
	where $\langle \sigma_u \sigma_v \rangle_{B_{\{u,v\}}(\ell)}^{f/+}$ represents the correlation in the Ising model on $B_{\{u,v\}}(\ell)$ with \emph{free} or $+$ boundary conditions.
	
	Note that every new vertex joins the graph with $m$ new edges. Therefore, $G_n$ has $n(m+o(1))$ edges. Hence, a uniformly chosen edge from $G_n$ can be viewed as a uniformly chosen out-edge from a uniformly random vertex. A finite neighbourhood of this uniformly chosen out-edge from a uniformly random vertex converges locally to the neighbourhood of the uniformly chosen out-edge of the root $\emp$ in the P\'{o}lya point tree. For any $\ell \geq 1$, the $\ell$-neighbourhood of a uniformly chosen edge in $G_n$ converges locally to the $\ell$-neighbourhood of $\hat{\emp}$ and $\emp1$, connected by an edge. Let us denote this tree as $\T(\ell)$. Consequently, as a result of local convergence and \cite[Lemma~6.4]{DM10}, for all $\ell \geq 1$, almost surely,
	\begin{equation}
		\label{for:lem:limiting:edge-measure:2}
		\lim_{n \to \infty} \E_n\Big[ \langle \sigma_u \sigma_v \rangle_{B_{(u,v)}(\ell)}^{f/+} \Big] = \E\Big[ \langle \sigma_u \sigma_v \rangle_{\T(\ell)}^{f/+} \Big].
	\end{equation}
	Now, using Lemmas~\ref{lem:tree-pruning} and \ref{lem:h:magnetization}, as $\ell \to \infty$,
	\begin{equation}
		\label{for:lem:limiting:edge-measure:3}
		\lim_{\ell \to \infty} \E\Big[ \langle \sigma_u \sigma_v \rangle_{\T(\ell)}^{f/+} \Big] = \E\Big[ \langle \sigma_u \sigma_v \rangle_{\nu_2^\prime} \Big],
	\end{equation}
	where $\nu_2^\prime(\sigma)$ is defined as
	\begin{equation}
		\label{for:lem:limiting:edge-measure:4}
		\nu_2^\prime(\sigma_1, \sigma_2) = \frac{1}{Z_2(\beta, \h(\hat{\emp}), \h(\hat{\emp}1))} \exp\Big[ \beta \sigma_1 \sigma_2 + \h(\hat{\emp}) \sigma_1 + \h(\hat{\emp}1) \Big].
	\end{equation}
	Simplifying the RHS of \eqref{for:lem:limiting:edge-measure:4}, we obtain
	\eqan{\label{for:lem:limiting:edge-measure:5}
		\E\big[\langle \sigma_1\sigma_2 \rangle_{\nu_2^\prime}\big]
		&=\E\left[ \frac{\e^{\beta+\h(\hat{\emp})+\h(\hat{\emp }1)}-\e^{-\beta-\h(\hat{\emp})+\h(\hat{\emp }1)}-\e^{-\beta+\h(\hat{\emp})-\h(\hat{\emp }1)}+\e^{\beta-\h(\hat{\emp})-\h(\hat{\emp }1)}}{\e^{\beta+\h(\hat{\emp})+\h(\hat{\emp }1)}+\e^{-\beta-\h(\hat{\emp})+\h(\hat{\emp }1)}+\e^{-\beta+\h(\hat{\emp})-\h(\hat{\emp }1)}+\e^{\beta-\h(\hat{\emp})-\h(\hat{\emp }1)}} \right]\nn\\
		&=\E\left[ \frac{\tanh(\beta)+\tanh(\h(\hat{\emp}))\tanh(\h(\hat{\emp }1))}{1+\tanh(\beta)\tanh(\h(\hat{\emp}))\tanh(\h(\hat{\emp }1))} \right]~,}
	thus proving \eqref{eq:lem:limiting:edge-measure-1}. Equation~\eqref{eq:lem:limiting:edge-measure-2} can be proved in a similar manner. Following the same steps as in \eqref{for:lem:limiting:edge-measure:1}--\eqref{for:lem:limiting:edge-measure:3} and using Lemma~\ref{lem:h:magnetization}, we can show that
	\begin{equation}
		\label{for:lem:limiting:edge-measure:6}
		\E_n[\langle \sigma_v \rangle_{\mu_n}] \overset{\prob}{\longrightarrow} \E\big[ \anglbrkt{\sigma_{\emp}}_{\mu} \big] = \E[\magn{\emp}] = \E[\tanh(\h(\emp))],
	\end{equation}
	thus completing the proof of the lemma.
\end{proof}

Lastly, we provide an alternative analytic proof of Lemma~\ref{lem:convexity:pressure-limit}. To prove this lemma, we show that the partial second derivatives of $\psi_n$ with respect to $\beta$ and $B$ are variances of some random variables, which essentially proves their non-negativity. %We defer this proof to Appendix.
\begin{proof}[Proof of Lemma~\ref{lem:convexity:pressure-limit}]
	We first perform the computation with respect to $\beta$; the computation for $B$ follows identically. For any $n \in \N$, differentiating $\psi_n$ with respect to $\beta$ gives
	\eqan{
		\frac{\partial}{\partial\beta}\psi_n(\beta,B)~=&~\frac{1}{n} \E_{\mu_n}\Big[ \sum\limits_{\{i,j\}\in E_n} \sigma_i\sigma_j \Big],\label{for:lem:convexity:1}\\
		\mbox{and}\quad \frac{\partial^2}{\partial\beta^2}\psi_n(\beta,B)~=&~\frac{1}{n}\Big[ \E_{\mu_n}\Big[ \Big( \sum\limits_{(i,j)\in E_n} \sigma_i\sigma_j \Big)^2 \Big]-\Big( \E_{\mu_n}\Big[ \sum\limits_{(i,j)\in E_n} \sigma_i\sigma_j \Big] \Big)^2 \Big]~,\label{for:lem:convexity:2}
	}
	where $\E_{\mu_n}$ denotes the expectation with respect to the $\mu_n$ measure. Note that the right-hand side of \eqref{for:lem:convexity:2} simplifies to ${\var}_{\mu_n}\left( \sum_{(i,j) \in E_n} \sigma_i \sigma_j \right)/n$. Therefore, $\frac{\partial^2}{\partial \beta^2} \psi_n(\beta, B)$ is non-negative almost surely, proving that $\beta \mapsto \psi_n(\beta, B)$ is a convex function. %Similarly, it can be shown that
	%\begin{equation}
	%	\frac{\partial^2}{\partial B^2} \psi_n(\beta, B) = \frac{1}{n} {\var}_{\mu_n} \left( \sum_{i \in [n]} \sigma_i \right). \label{for:lem:convexity:3}
	%\end{equation}
	%Hence, using a similar argument, we obtain that $B \mapsto \psi_n(\beta, B)$ is also a convex function.
\end{proof}




\end{appendices}
%\bibliographystyle{acm}
\bibliography{bibliofile}% common bib file

\end{document}